\documentclass[10pt]{article}
\usepackage{amsmath,amsfonts,amssymb,amsthm}
\usepackage[utf8]{inputenc}
\usepackage{graphicx}
\usepackage{color}
\usepackage{tikz}
\usepackage{array}
\usepackage{url}
\usepackage{subcaption}
\usepackage{mathtools}

\usetikzlibrary{patterns,math,arrows.meta}
\usepackage{hyperref}
\usepackage{cleveref}

\usepackage{titling}
\makeatletter
\pgfdeclarepatternformonly{south east lines}{\pgfqpoint{-0pt}{-0pt}}{\pgfqpoint{3pt}{3pt}}{\pgfqpoint{3pt}{3pt}}{
    \pgfsetlinewidth{0.4pt}
    \pgfpathmoveto{\pgfqpoint{0pt}{3pt}}
    \pgfpathlineto{\pgfqpoint{3pt}{0pt}}
    \pgfpathmoveto{\pgfqpoint{.2pt}{-.2pt}}
    \pgfpathlineto{\pgfqpoint{-.2pt}{.2pt}}
    \pgfpathmoveto{\pgfqpoint{3.2pt}{2.8pt}}
    \pgfpathlineto{\pgfqpoint{2.8pt}{3.2pt}}
    \pgfusepath{stroke}}
\makeatother

\newdimen\GridSize
\tikzset{
	GridSize/.code={\GridSize=#1},
	GridSize=3pt
}

\newtheorem{theorem}{Theorem}[section]
\newtheorem{definition}{Definition}[section]
\newtheorem{proposition}[theorem]{Proposition}

\newtheorem{assumption}[theorem]{Assumption}
\theoremstyle{definition}
\newtheorem{example}[definition]{Example}
\newtheorem{remark}[theorem]{Remark}

\DeclareMathOperator{\rank}{rank}

\newcommand{\R}{\mathbb{R}}
\newcommand{\C}{\mathbb{C}}

\newcommand{\ie}{i.e.~}
\newcommand{\minus}{\scalebox{0.9}[1.0]{\( \,-\, \)}}

\title{Analysis and Design of Strongly Stabilizing PID Controllers for Time-Delay Systems}
\author{
{Pieter Appeltans}\\
{KU Leuven, Department of Computer Science, NUMA Section}\\
{Leuven, Belgium.}\\
{Email: {\tt Pieter.Appeltans@cs.kuleuven.be}} \medskip\\
{Silviu-Iulian Niculescu}\\
{Universit\'{e} Paris-Saclay, CNRS, CentraleSup\'{e}lec,}\\ {Laboratory of Signals and Systems (L2S), Inria Team ``DISCO''}\\ 
{Gif-sur-Yvette, France.}\\
{E-mail: {\tt Silviu.Niculescu@l2s.centralesupelec.fr}} \medskip \\
{Wim Michiels}\\
{KU Leuven, Department of Computer Science, NUMA Section}\\
{Leuven, Belgium.} \\
{Email: {\tt Wim.Michiels@cs.kuleuven.be}} \medskip\\
}

\date{}

\begin{document}
\maketitle
\begin{abstract}
    This paper presents the analysis of the stability properties of PID controllers for dynamical systems with multiple state delays, focusing on the mathematical characterization of the potential sensitivity of stability with respect to infinitesimal parametric perturbations. These perturbations originate for instance from neglecting feedback delay, a finite difference approximation of the derivative action, or neglecting fast dynamics. The analysis of these potential sensitivity problems leads us to the introduction of a `robustified' notion of stability called \emph{strong stability}, inspired by the corresponding notion for neutral functional differential equations. We prove that strong stability can be achieved by adding a low-pass filter with a sufficiently large cut-off frequency to the control loop, on the condition that the filter itself does not destabilize the nominal closed-loop system. Throughout the paper, the theoretical results are illustrated by examples that can be analyzed analytically, including, among others, a third-order unstable system where both proportional and derivative control action are necessary for achieving stability, while the regions in the gain parameter-space for stability and strong stability are not identical. Besides the analysis of strong stability, a computational procedure is provided for designing strongly stabilizing PID controllers.  Computational case-studies illustrating this design procedure complete the presentation.
\end{abstract}
{
\small
\textbf{\textit{Keywords ---}} PID control, Stability, Modelling uncertainty, Robustness. \smallskip \\
\textbf{\textit{AMS subject classifications} ---} 93B35, 93B52,93C23,93D15,93D22.
}
\section{Introduction}
In this paper we analyze proportional-integral-derivative (PID) output feedback control of multiple-input-multiple-output (MIMO) linear time invariant (LTI) dynamical systems with discrete state delays. PID controllers are used in many control applications and are well-established in industry due to their implementation simplicity. Literature extending the PID control framework to systems with delays includes, among others, \cite{silva2007,Gundes2007,Ozbay2006,Morarescu2011,Fiser2018,Zitek2019,Zitek2020,Rodriguez2019,Ma2019} and the references therein, and focuses on an analytic characterization of the stabilization and stabilizability of (low-order) systems with input/output delay.

Central in the presented work is the analysis of the sensitivity of stability of the closed-loop system with respect to arbitrarily small modelling errors, which include for instance neglected feedback delay, a finite-difference approximation of the derivative and neglected high-frequency behavior. Similar sensitivity problems are examined in \cite{Georgiou1989}, in which the concept of \emph{w-stability\/} is introduced. w-Stability requires that the closed-loop system remains stable for sufficiently small perturbations, where the considered perturbation class consists of approximate identities (see \cite[Section V.C]{Zames1981} for a definition of approximate identities). Of interest for this work, is Section 4 which examines non-proper rational controllers, such as PID controllers, for the control of rational (finite-dimensional) plants. More specifically, Theorem 2 in \cite{Georgiou1989} suggests that  w-stability can be induced by adding a low-pass filter with a sufficiently large cut-off frequency to the control loop. However, as suggested by \cite[Theorem 3]{Georgiou1989}
such a low-pass filter can itself destroy the stability of the nominal closed-loop system. In this work, we will derive similar results but we will focus on time-delay systems, which are infinite dimensional, and PID control. Furthermore, we will use a different perturbation class and besides perturbations on the input and output channels, the framework presented here also includes a perturbation inside the (derivative part of the) controller. This additional perturbation is motivated by the implementation of the derivative using a finite-difference approximation.

 Another important starting point is article~\cite{Michiels2009}, which examines the stabilizability of a controllable LTI system using state-derivative feedback control (i.e.~using a control law of the form $u(t)=K \dot x(t)$ instead of the conventional state feedback $u(t)=K x(t))$. While the nominal system can be stabilized if and only if the system matrix of the open-loop system has no zero eigenvalue, the stability can be destroyed by infinitesimal perturbations if the open-loop system matrix has an odd number of eigenvalues in the open-right half plane (the so-called odd number limitation). However, if this number  of eigenvalues is even, then robust stability in the presence  of sufficiently small perturbations can always be achieved by adding a low-pass filter to the control law.  

Finally, there exists an abundant literature that addresses fragility problems with respect to infinitesimal perturbations,
that include, among others, \cite{Logemann1996b,datko1991,Hale1995,Morgul1995,Logemann1998} (and the references therein) discussing the sensitivity of stability of controlled systems with respect to feedback delay, \cite{Hale2002,Michiels2002,Logemann1996} (and the references therein) discussing the sensitivity of stability of neutral functional differential equations with respect to infinitesimal (changes on the) delays, and \cite{Sipahi2006} discussing the fragility of a gradient play dynamics model against a derivative approximation.
%
%


In this work we will extend the theoretical framework of \cite{Michiels2009} to systems with state delays controlled by PID output feedback, and complement it with an algorithm for the design of the controller gain matrices. More specifically, after introducing the considered set-up and illustrating the aforementioned sensitivity problems in \Cref{sec:problem_statement}, we derive conditions on the derivative gain matrix under which a stable closed-loop system loses stability under arbitrarily small feedback delay (\Cref{subsec:losing_stability_feedback_delay}) and under an arbitrary close finite-difference approximation of the derivative (\cref{subsec:losing_stability_finite_difference}). In \Cref{sec:strong_stability}, we introduce the notion of \emph{strong stability}, in analogy with the corresponding notion for neutral functional differential equations \cite{Hale2002}, which requires the closed-loop system to remain stable under sufficiently small perturbations. As in \cite{Michiels2009}, we will show that a strongly stable closed-loop can be obtained by adding a low-pass filter to the control loop, under the condition that this low-pass filter does not destabilize the system, which induces an algebraic constraint on the derivative gain matrix. Subsequently, \Cref{sec:numerical_algorithm} presents a computational procedure to design strongly stabilizing PID controllers with low pass filtering. This procedure consists of minimizing the spectral abscissa of the nominal (\ie{}without low-pass filter) closed-loop system in function of the elements of the feedback matrices under the aforementioned constraint on the derivative gain matrix. The presented method thus fits in the direct optimization framework as, for example, used in \cite{Michiels2011,Dileep2018}. This framework has the advantage that a particular structure of the feedback matrices can be easily incorporated in the design procedure. This comes, however, at the cost of having to solve a small non-convex non-smooth optimization problem. \Cref{sec:generalization} presents some remarks on how to generalize the presented results to systems with input delay and to systems with bounded uncertainties on the system matrices and the state delays. Finally, \Cref{sec:conclusion} concludes the paper.

\section{Problem statement}
\label{sec:problem_statement}
This work considers LTI systems with discrete state delays of the form: 
\begin{equation}
\label{eq:sys}
\left\{\begin{array}{l}
\dot x(t) = A_0\, x(t) + \sum_{k=1}^{K} A_k\, x(t-\tau_k) +B\, u(t),\\
 y(t)=C\, x(t) 
\end{array}\right.
\end{equation}
with $x\in\mathbb{R}^n$ the internal state, $\ u\in\mathbb{R}^m$ the input,$\ y\in\mathbb{R}^p$ the output, $0<\tau_1<\dots<\tau_K<+\infty$ discrete delays, $A_0$, $A_1$, \dots, $A_K \in \R^{n\times n}$, $B\in \R^{n\times m}$ and $C\in\R^{p\times n}$. Furthermore, without loss of generality, we will assume that $B$ is of full column rank and that $C$ is of full row rank. The goal is to design a PID output feedback control law
\begin{equation}
\label{eq:pid_control}
u(t)= K_p\, y(t)+ K_d\, \dot y(t) + K_i \int_0^{t} y(s) ds,
\end{equation}
with $K_p$, $K_d$ and $K_i$ real-valued $m\times p$ matrices, that exponentially stabilizes the system. We will examine the exponential stability of the closed-loop system consisting of \eqref{eq:sys} and \eqref{eq:pid_control} using the following characteristic function
\begin{equation}
\label{eq:characteristic_function}
H_0(\lambda) :=
\det
\left(
\lambda
\begin{bmatrix}
I_n-BK_d C & 0\\
0 & I_q 
\end{bmatrix}-
\begin{bmatrix}
A_0+BK_p C & BU_i \\
V_i C & 0
\end{bmatrix}-
\sum_{k=1}^{K} \begin{bmatrix}
A_k & 0 \\
0 & 0
\end{bmatrix}e^{-\lambda \tau_k}
\right)
\end{equation}
with $I_n$ the identity matrix of size $n$, $q=\rank(K_i)$ and $K_i = U_i V_i$, a rank revealing decomposition with $U_i\in \R^{m\times q}$ of full column rank and $V_i \in \R^{q\times p}$ of full row rank. More specifically, the closed-loop system is exponentially stable if and only if all roots of \eqref{eq:characteristic_function} lie inside the open left half-plane bounded away from the imaginary axis. However, the following example illustrates that a stable closed-loop system can loose stability due to arbitrarily small implementation errors.

\begin{example}
\label{ex:2nd_order}
Consider the following second-order, single-input single-output system:
\begin{equation*}
\label{example_second_order_plant}
\left\{
\begin{array}{rcl}
\dot{x}(t) &=& \begin{bmatrix}
0 & \omega_0 \\
\omega_0 & 0
\end{bmatrix}\ x(t) + \begin{bmatrix}
-1 \\
0
\end{bmatrix}\ u(t)\\[1em]
y(t) &=& \begin{bmatrix}
1 & 0
\end{bmatrix}\ x(t)
\end{array}
\right.
\end{equation*}
with as corresponding transfer function 
$
-s/(s^2-\omega_0^2).
$
We want to control this system using the feedback law $u(t)=k_p y(t)+k_d \dot{y}(t)$, that is nothing else than a PD controller. For this set-up, characteristic function \eqref{eq:characteristic_function} reduces to the following second-order polynomial:
\begin{equation*}
(1+k_d)\lambda^2+k_p \lambda-\omega_0^2.
\end{equation*}
It follows from the Routh-Hurwitz stability criterion that the roots of this polynomial lie inside the open left half-plane if (and only if) $k_d<-1$ and $k_p<0$. The system is thus stabilized by any control law with $k_d<-1$ and $k_p<0$. Next, we examine the effect of delayed feedback and a finite-difference approximation of the derivative on the stability of the closed-loop and we will show that any stable closed-loop system loses stability under these implementation errors, even when the size of the error is arbitrarily small.

As first implementation error, we consider delayed feedback, \ie{}$u(t)=k_p y(t-r)+k_d \dot{y}(t-r)$ with $r>0$. In this case, the exponential stability of the closed-loop system is characterized by the roots of the following quasi-polynomial
\begin{equation}
\label{eq:example_2nd_order_feedback_delay}
(1+k_d e^{-\lambda r})\lambda^2 + k_p \lambda e^{-\lambda r} - \omega_0^2.
\end{equation}
It follows from \cite[Proposition 1.28]{Michiels2014} that  \eqref{eq:example_2nd_order_feedback_delay} has a sequence of roots $\{\lambda_k\}_{k=1}^{\infty}$ which satisfies
\[
\lim\limits_{k\to \infty} |\Im(\lambda_k)| = +\infty \text{ and } \lim\limits_{k\to \infty} \Re(\lambda_k) = \frac{1}{r}\ln\left(|k_d|\right).
\]
Characteristic function \eqref{eq:example_2nd_order_feedback_delay} thus has (infinitely many) roots in the closed right half-plane for any $k_d < -1$ and any $r>0$. Under delayed feedback, the closed-loop system is thus no longer stable, even if the feedback delay is arbitrarily small. 
As second implementation error, we consider a finite-difference approximation of the derivative:
\begin{equation*}
\label{finite_difference_approximate}
\dot{y}(t) \approx \frac{y(t)-y(t-r)}{r},
\end{equation*}
with $r>0$. Such an approximation might be necessary if there is no sensor to measure the actual derivative. For the considered example, the exponential stability of the closed-loop system is now characterized by the roots of
\begin{equation}
\label{eq:example_2nd_order_finite_difference}
f(\lambda) := \lambda^2 + \left(k_p + k_d \frac{1-e^{-\lambda r}}{r}\right) \lambda - \omega_0 ^2.
\end{equation}
This characteristic function can be rewritten as $f(\lambda) = \lambda^{2} \Big(f_1(\lambda)-f_2(\lambda) \Big)$ with
\[
f_1(\lambda) :=  1 + \frac{k_p}{\lambda} - \frac{\omega_0^2}{\lambda^2}
\text{\ \ and\ \ } 
f_2(\lambda) := -k_d  \frac{1-e^{-\lambda r}}{\lambda r}.
\]
As $\lambda = 0$ is not a root of $f(\cdot)$, the roots of \eqref{eq:example_2nd_order_finite_difference} thus correspond to the crossings of $f_1(\cdot)$ and $f_2(\cdot)$ in the complex plane.  \Cref{fig:introduction_example} shows these functions for $\lambda$ real-valued and positive, $k_d<-1$, $k_p<0$ and $r>0$. In this case, characteristic function \eqref{eq:example_2nd_order_finite_difference} thus has a real-valued root in the right half-plane for all $r>0$. Furthermore, this root  moves to $+\infty$ as $r\to0+$. So also under an arbitrary accurate finite-difference approximation of the derivative, stability is lost.  \hfill $\diamond$
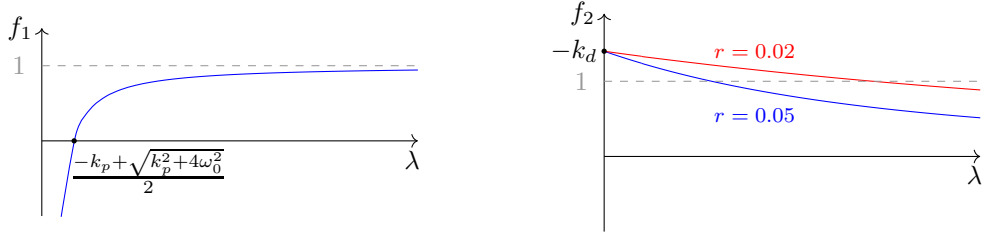
\begin{figure}[!htbp]
\begin{minipage}[c]{0.48\textwidth}
	\begin{tikzpicture}
		\draw[->] (0, 0) -- (5, 0) node[below,xshift=-0.5ex] {$\lambda$};
		\draw[->] (0, -1) -- (0, 1.5) node[left] {$f_1$};
		\draw[xscale=0.25,domain=1.03:20, smooth, variable=\x, blue] plot ({\x}, {1-1.1/\x-1/(\x*\x)});
		\fill[fill=black] (0.43,0) circle (1pt);
		\node[right] (label) at (0.25,-0.4) {$\frac{-k_p+\sqrt{k_p^2+4\omega_0^2}}{2}$};
		\draw[dashed,gray!80] (0,1) -- (5,1);
		\node[gray!80] () at (-0.3,1) {1};
	\end{tikzpicture}
\end{minipage}
\begin{minipage}[c]{0.48\textwidth}
	\begin{tikzpicture}
		\draw[->] (0, 0) -- (5, 0) node[below,xshift=-0.5ex] {$\lambda$};
		\draw[->] (0, -1) -- (0, 1.9) node[left] {$f_2$};
		\draw[xscale=0.1,domain=0:5, smooth, variable=\x, blue] plot ({\x}, {1.4*(1-(\x*0.05)/2+(\x*0.05)*(\x*0.05)/6});
		\draw[xscale=0.1,domain=0:5, smooth, variable=\x, red] plot ({\x}, {1.4*(1-(\x*0.02)/2+(\x*0.02)*(\x*0.02)/6});
		\draw[xscale=0.1,domain=5:50, smooth, variable=\x, blue] plot ({\x}, {1.4*((1-exp(-\x*0.05)))/(\x*0.05)});
		\draw[xscale=0.1,domain=5:50, smooth, variable=\x, red] plot ({\x}, {1.4*((1-exp(-\x*0.02)))/(\x*0.02)});
		\draw[dashed,gray!80] (0,1) -- (5,1);
		\node[gray!80] () at (-0.3,1) {1};
		\fill[fill=black] (0,1.4) circle (1pt);
		\node[] () at (-0.4,1.4) {$-k_d$};
		\node[blue] () at (2,0.55) {\footnotesize$r =0.05$};
		\node[red] () at (2,1.4) {\footnotesize$r =0.02$};
	\end{tikzpicture}
\end{minipage}
\caption{The functions $f_1(\lambda)$ and $f_2(\lambda)$ for $\lambda$ real-valued and positive, $k_d<-1$, $k_p<0$ and $r>0$.}
\label{fig:introduction_example}
\end{figure}

\end{example}
\begin{remark}
	We consider \eqref{eq:characteristic_function} as characteristic function instead of 
	\[
	\det
	\left(
	\lambda
	\begin{bmatrix}
	I_n-BK_d C & 0\\
	0 & I_p 
	\end{bmatrix}-
	\begin{bmatrix}
	A_0+BK_p C & BK_i \\
	C & 0
	\end{bmatrix}
	-\sum_{k=1}^{K}
	\begin{bmatrix}
	A_k & 0 \\
	0 & 0
	\end{bmatrix}e^{-\lambda \tau_k}
	\right),
	\]
	which one would naively obtain by bringing the characteristic equation associated with the closed-loop system to linear form. This is motivated by the fact that this last equation might introduce roots in the origin that have no relation with with the actual asymptotic behavior of the closed-loop system. Furthermore, note that $q$ $\big(=\rank(K_i)\big)$ corresponds to the minimal number of integrators needed to implement the (integral) control action. \hfill $\diamond$
\end{remark}

\section{Losing stability under arbitrarily small implementation errors}
\label{sec:losing_stability}
Motivated by \Cref{ex:2nd_order}, this section will derive conditions on the spectrum of $BK_{d}C$ under which stability is lost under feedback delay (\Cref{subsec:losing_stability_feedback_delay}) and a finite-difference approximation of the derivative (\Cref{subsec:losing_stability_finite_difference}), even when the size of the error is arbitrarily small.

\subsection{Feedback delay}
\label{subsec:losing_stability_feedback_delay}
For studying the effect of delayed feedback, on a stabilizing feedback law of form \eqref{eq:pid_control}, consider now the following system:
\begin{equation}
\label{eq:feedback_delay}
	\dot{x}(t) = A_0\, x(t) + \sum_{k=1}^{K} A_k\, x(t- \tau_k) + B\, u(t-r),
\end{equation}
derived from (\ref{eq:sys}) by including a delay $r>0$ in the input. 

More precisely, the following proposition gives a \emph{sufficient condition\/} on the spectrum of $BK_{d}C$ under which stability is lost under delayed feedback with an arbitrarily small delay.
\begin{proposition}
	\label{prop:delayed_feedback}
	Assume that the gain matrices $K_p$, $K_d$ and $K_i$ are such that the closed-loop system \eqref{eq:sys}-\eqref{eq:pid_control} is exponentially stable. If the spectral radius of $B K_d C$ is larger than one, then the closed-loop of \eqref{eq:feedback_delay} and \eqref{eq:pid_control} is unstable for all $r>0$.
\end{proposition}
\begin{proof}
	Under delayed feedback, the closed-loop system is unstable if 
	\begin{equation}
	\label{eq:characeristic_function_delayed_feedback}
\begin{array}{>{$}p{.9\textwidth}<{$}}
	 H_{1}(\lambda;r) :=\det\Bigg(\lambda \left(\begin{bmatrix}
	I_n & 0\\
	0 & I_q
	\end{bmatrix}- \begin{bmatrix}
	BK_d C  & 0 \\
	0 & 0
	\end{bmatrix}e^{- \lambda r} \right)- \begin{bmatrix}
	A_0 & 0 \\
	V_i C & 0
	\end{bmatrix} \\ \hfill - 
	\begin{bmatrix}
	BK_p C  & B U_i \\
	0 & 0
	\end{bmatrix} e^{- \lambda r}-
	\displaystyle\sum_{k=1}^{K}
	\begin{bmatrix}
	A_k & 0 \\
	0 & 0
	\end{bmatrix}e^{- \lambda \tau_k} \Bigg)
	\end{array}
	\end{equation}
	has at least one root inside the right half-plane. Note that \eqref{eq:characeristic_function_delayed_feedback} corresponds to the characteristic function of a neutral delay eigenvalue problem \cite[Section~1.2]{Michiels2014}. It follows therefore from  \cite[Proposition 1.28]{Michiels2014} that for non-zero $BK_d C$, the characteristic function \eqref{eq:characeristic_function_delayed_feedback} has a sequence of characteristic roots $\{\lambda_k\}_{k=1}^{\infty}$ for which
	\[
	\lim\limits_{k\to \infty} |\Im(\lambda_k)| = +\infty \text{ and } \lim\limits_{k\to \infty} \Re(\lambda_k) = \frac{1}{r}\ln\big(\rho(BK_d C)\big)
	\]
	with $\rho(\cdot)$ the spectral radius of its matrix argument. If $\rho(BK_dC)> 1$, then for all $r>0$, characteristic function $H_1(\cdot;r)$ has thus (infinitely many) characteristic roots in the right half-plane. 
\end{proof}

\subsection{Approximation of derivatives}
\label{subsec:losing_stability_finite_difference}
Next, we derive a similar condition for the effect of using a finite-difference approximation instead of the actual derivative. The feedback signal now becomes:
\begin{equation}
	\label{eq:pid_finite_difference}
	u(t) = K_p\, y(t) + K_d\, \frac{y(t)-y(t-r)}{r} + K_i \int_0^{t} y(s)\, ds.
\end{equation} 
\begin{proposition}
	\label{prop:finite_difference}
	Assume that the gain matrices $K_p$, $K_d$ and $K_i$ are such that the closed-loop system \eqref{eq:sys}-\eqref{eq:pid_control} is exponentially stable. If at least one of the eigenvalues of $B K_d C$ lies outside $\mathrm{clos}(S)$, with 
	\begin{equation*}\label{defS}
	S:=
	\left\{\begin{array}{l}
	\lambda\in\mathbb{C}:\ \Im(\lambda)\in (-\pi,\ \pi)\mathrm{\ and\ } \\
	\hspace*{1.6cm}\Re(\lambda) <\left\{\begin{array}{ll}
	\Im(\lambda)\ \mathrm{cotan}(\Im(\lambda)), & \Im(\lambda)\in(-\pi,\ 0)\cup(0,\ \pi),\\
	1, & \Im(\lambda)=0,
	\end{array}\right.
	\end{array}\right\}
	\end{equation*}
	then there exists a $\hat{r}>0$ such that the closed-loop of \eqref{eq:sys} and \eqref{eq:pid_finite_difference} is unstable for all $r\in(0,\hat{r})$
\end{proposition}
\begin{proof}
	The stability of the closed-loop of \eqref{eq:sys} and \eqref{eq:pid_finite_difference} is characterized by the roots of 
	\begin{equation}
	\label{eq:characteristic_function_finite_difference}
	\begin{array}{>{$}p{.9\textwidth}<{$}}
	H_2(\lambda;r) := \det\Bigg(\lambda\left(
	\begin{bmatrix}
	I_n & 0 \\
	0 & I_q
	\end{bmatrix} -\begin{bmatrix}
	 BK_d C & 0 \\
	 0 & 0
	\end{bmatrix}\frac{1-e^{-\lambda r}}{\lambda r}\right) \\ \hfill - \begin{bmatrix}
	A + BK_p C  & BU_i \\
	V_i C & 0
	\end{bmatrix}- \displaystyle\sum_{k=1}^{K}
	\begin{bmatrix}
	A_k & 0 \\
	0 & 0
	\end{bmatrix} e^{-\lambda \tau_k}
	\Bigg).
	\end{array}
	\end{equation}
	Multiplying $H_2(\cdot;r)$ with $r^{n+q}$ and introducing the change of variables $\widehat{\lambda} = \lambda r$, we obtain
	\[
	\begin{array}{>{$}p{.9\textwidth}<{$}}
	G(\widehat{\lambda};r) := \det\Bigg(\widehat{\lambda}
	\begin{bmatrix}
	I_n & 0 \\
	0 & I_q
	\end{bmatrix} -\begin{bmatrix}
	BK_d C & 0 \\
	0 & 0
	\end{bmatrix}\big(1-e^{-\widehat{\lambda}}\big) \\ \hfill - r \begin{bmatrix}
	A + BK_p C  & BU_i \\
	V_i C & 0
	\end{bmatrix} - \displaystyle\sum_{k=1}^{K} \begin{bmatrix}
	A_k & 0 \\
	0 & 0
	\end{bmatrix} r e^{-\widehat{\lambda} (\tau_k/r)}
	\Bigg).
	\end{array}
	\]
	As $r\searrow 0$, $G(\cdot;r)$ uniformly converges to
	\[
	\ \widetilde{G}(\widehat{\lambda}) := \det\left(
	\widehat{\lambda} \begin{bmatrix}
	I_n & 0 \\
	0 & I_q
	\end{bmatrix}-
	\begin{bmatrix}
	BK_d C & 0 \\
	0 & 0
	\end{bmatrix} 
	\big(1- e^{-\widehat{\lambda}}\big)
	\right)
	\]
	on compact regions in the open right half-plane. This function $\widetilde{G}(\cdot)$ has a root of multiplicity (at least) $r$ at the origin; the other roots of $\widetilde{G}(\cdot)$ are the solutions of the following $n$ equations
	\[
	\widehat{\lambda} - \lambda_i \left(1-e^{-\widehat{\lambda}}\right) = 0\quad \text{for }i=1,\dots,n,
	\]
	in which $\{\lambda_i\}_{i=1}^{n}$ are the eigenvalues of $BK_d C$. It follows from \cite[Lemma A.1]{Michiels2009} that $\widetilde{G}(\cdot)$ has a root in the open right half-plane if (and only if) $BK_d C$ has an eigenvalue outside $\mathrm{clos}(S)$. As $G(\cdot;r)$ uniformly converges to $\widetilde{G}(\cdot)$ on compact regions in the open right half-plane, it follows from a similar argument as in \cite[Proposition~3.1 - Case~1]{Michiels2009} that there exist $c>0$ and $\hat{r}>0$ such that the function $G(\cdot;r)$ has at least one root in the right half-plane $\{\lambda \in \C:\Re(\lambda)> c \}$ for all $r \in (0,\hat{r})$. Because the roots of $H_2(\cdot;r)$ correspond to the roots of $G(\cdot;r)$ divided by $r$, $H_2(\cdot;r)$ has at least one root in the right half-plane $\{\lambda \in \C:\Re(\lambda)> c/r \}$. 
\end{proof}

\section{Strong stability}
\label{sec:strong_stability}
The previous section demonstrated that, under certain conditions on the spectrum of $BK_d C$, there exist arbitrarily small implementation errors that render a stable closed-loop system unstable. Or in other words, under certain conditions on the derivative gain matrix, the stability of the closed-loop system is \emph{sensitive} with respect to infinitesimal perturbations. As observed in the proofs of the previous section, these perturbations introduced characteristic roots in the right half-plane that move to $+\infty$ as the perturbation size goes to zero. An intuitive way to avoid such right half-plane roots is to apply a low-pass filter to the derivative signal, leading to the following control law:
\begin{equation}
\label{eq:pid_lowpass_filter}
u(t) = K_p y(t) + K_d\zeta(t) + K_i \int_{0}^{t} y(s) ds 
\end{equation}
with
\[
T \dot{\zeta}(t) + \zeta(t) =  \dot{y}(t),
\]
and $1/T$ the cut-off frequency. However, as observed in \cite[Theorem~3]{Georgiou1989} and \cite[Secion~4.1]{Michiels2009} this low-pass filter might itself be destabilizing. 

The following subsection shows that this low-pass filter does not destroy stability for sufficiently large cut-off frequencies if a certain algebraic constraint on $K_d$ is fulfilled. Subsequently, \Cref{subsec:strong_stability} introduces the notion of strong stability, which requires the closed-loop to remain stable under sufficiently small perturbations. In that subsection we will also show that if the aforementioned constraint on $K_d$ is fulfilled, then the closed-loop with low-pass filtering is strongly stable. Finally, these theoretical results and those of the previous section are illustrated using a third-order system in \Cref{subsec:illustration_third_order}.
\subsection{Losing stability under low-pass filtering}
\label{subsec:low_pass_filter}
The following proposition shows that, under a particular condition on the spectrum of $B K_d C$, the stability of the closed-loop system \eqref{eq:sys}-\eqref{eq:pid_control} is sensitive with respect to low pass filtering of the derivative signal. However, this proposition also shows that if the real parts of the eigenvalues of $B K_d C$ are smaller than 1, then stability of the closed-loop system is preserved when replacing control law \eqref{eq:pid_control} by control law \eqref{eq:pid_lowpass_filter} with $T$ sufficiently small.
\begin{proposition}
	\label{prop:low_pass}
	Assume that the gain matrices $K_p$, $K_d$ and $K_i$ are such that the closed-loop system \eqref{eq:sys}-\eqref{eq:pid_control} is exponentially stable. If $BK_dC$ has an eigenvalue with real part larger than one, then there exists a $\widehat{T}>0$ such that the closed loop \eqref{eq:sys}-\eqref{eq:pid_lowpass_filter} is unstable for all $T \in (0,\widehat{T})$. On the other hand, if $BK_dC -I_n$ is Hurwitz then there exists a $\widehat{T}>0$ such that the closed loop \eqref{eq:sys}-\eqref{eq:pid_lowpass_filter} is exponentially stable for all $T \in (0,\widehat{T})$.
\end{proposition}
\begin{proof}
	The closed-loop stability of system \eqref{eq:sys}-\eqref{eq:pid_lowpass_filter} is characterized by the roots of
	\begin{equation}
	\label{eq:characteristic_function_lowpass_filter}
	\begin{array}{>{$}p{.9\textwidth}<{$}}
	H_3(\lambda;T) := \det
	\Bigg(
	\lambda
	\left(
	\begin{bmatrix}
	I_n & 0\\
	0 & I_q 
	\end{bmatrix}- \frac{1}{\lambda T+1} \begin{bmatrix}
	BK_d C & 0 \\
	0 & 0
	\end{bmatrix}\right) \\[10pt] \hfill -
	\begin{bmatrix}
	A_0+BK_p C & BU_i \\
	V_i C & 0
	\end{bmatrix}-
	\displaystyle\sum_{k=1}^{K}
	\begin{bmatrix}
	A_k & 0 \\
	0 & 0
	\end{bmatrix}e^{-\lambda \tau_k}
	\Bigg).
	\end{array}
	\end{equation}

	If $BK_d C$ has an eigenvalue with real part larger than one, then one can use a similar derivation as in the proof of \Cref{prop:finite_difference} to show that there exist $c>0$ and $\widehat{T}>0$ such that $H_3(\cdot;T)$ has a root in the right half-plane $\{\lambda \in \C:\Re(\lambda)> c/T \}$ for all $T\in(0,\widehat{T})$.

	On the other hand, if $BK_d C-I_n$ is Hurwitz, then for each $\epsilon>0$ there exists a $\widehat{T}>0$ such that $
	I_n- (BK_{d} C)/(\lambda T + 1)
	$ is invertible for all $\lambda$ that lie inside the right half-plane \linebreak $V:= \left\{\lambda\in\C:\Re(\lambda) > -\epsilon \right\}$ and all $T\in(0,\widehat{T})$. This means that inside the right half-plane $V$, $H_3(\lambda;T)=0$ is equivalent with
	\[
	\det\Bigg( \lambda\begin{bmatrix}
	I_n & 0\\
	0 & I_q 
	\end{bmatrix}-
	\begin{bmatrix}
	\left(I_n- \frac{BK_d C}{\lambda T+1}\right)^{-1} & 0\\
	0 & I_q
	\end{bmatrix}  \left(
	\begin{bmatrix}
	A_0+BK_p C & BU_i \\
	V_i C & 0
	\end{bmatrix}+
	\displaystyle\sum_{k=1}^{K}
	\begin{bmatrix}
	A_k & 0 \\
	0 & 0
	\end{bmatrix}e^{-\lambda \tau_k}\right) \Bigg) = 0.
	\]
	A characteristic root, $\lambda_0$, of $H_3(\cdot;T)$ inside $V$ is thus bounded in modulus by
	\[
	|\lambda_0| \leq \sup_{\lambda \in V } \left\|\begin{bmatrix}
	\left(I_n- \frac{BK_d C}{\lambda T+1}\right)^{-1} & 0\\
	0 & I_q 
	\end{bmatrix}\right\| \left(
	\left\|	\begin{bmatrix}
	A_0+BK_p C & BU_i \\
	V_i C & 0
	\end{bmatrix}\right\| + \sum_{k=1}^{K} \|A_k\| e^{-\Re(\lambda) \tau_k} 
	\right) := \Xi < \infty
	\]
	Because $H_3(\cdot;T)$ uniformly converges to $H_0(\cdot)$ on compact regions in the complex plane as $T\searrow0$, there exists a $\widetilde{T}\leq \widehat{T}$ such that for all $T\in(0,\widetilde{T})$
	\[
	\max_{\lambda\in\partial\Omega} |H_3(\lambda;T)-H_0(\lambda)| \leq \min_{\lambda \in \partial\Omega} |H_0(\lambda)|
	\]
	with $\Omega$ the intersection of $V$ and $\{\lambda\in\C:|\lambda|\leq \Xi \}$. It now follows from Rouch\'e's theorem \cite{Titchmarsh1939} that 
	$H_3(\cdot;T)$ and $H_0(\cdot)$ have the same number of zeros in $\Omega$, and hence also in $V$, for all $T\in (0,\widetilde{T})$. The proposition now follows by choosing $-\epsilon$ larger than the real part of the right-most root of $H_0(\cdot)$.
\end{proof}

\subsection{Strongly stable closed-loop}
\label{subsec:strong_stability}
	In this subsection, we will show that if the low-pass filter itself is not destabilizing, then the closed-loop system \eqref{eq:sys}-\eqref{eq:pid_lowpass_filter} does not suffer from the same sensitivity problems as the ones encountered in \Cref{sec:losing_stability}. Furthermore, instead of restricting ourselves to delayed feedback and a finite-difference approximation of the derivative, we will consider all perturbations that fit the framework given in \Cref{fig:strong_stability_framework}, in which perturbation functions $R_1(\cdot;\cdot) : \C \times  [0,+\infty) \mapsto \C^{m\times m}$, $R_2(\cdot;\cdot) : \C \times  [0,+\infty) \mapsto \C^{p \times p}$ and $R_3(\cdot;\cdot) : \C \times  [0,+\infty) \mapsto \C^{p\times p}$ fulfill the following assumptions:
	\begin{assumption}~
	\label{ass:perturbations}
	\begin{enumerate}
		\item for every $r \geq0$, the functions $\left\{\lambda\mapsto R_{i}(\lambda;r)\right\}_{i=1}^3$ are meromorphic and for every $\lambda \in \C$, the functions $\left\{r \mapsto R_{i}(\lambda;r)\right\}_{i=1}^3$ are continuous;
		\item the matrices $R_{1}(\lambda;r)$ and $R_{3}(\lambda;r)$ are of full rank for all $\lambda\in \mathbb{C}$ and all $r \geq 0$;
		\item $\lambda \mapsto R_{i}(\lambda;0) \equiv I$ for  $i=1,2,3$;
		\item for every compact set $\Omega \subset \C$, we have
		\[
		\lim\limits_{r\rightarrow 0+} \max_{\lambda \in \Omega} \|R_{i}(\lambda;r) - I\| = 0 \text{ for }i=1,2,3
		\]\ie{}the functions $\left\{\lambda\mapsto R_{i}(\lambda;r)\right\}_{i=1}^{3}$ uniformly converge to the identity matrix on compact regions in the complex plane as $r$ goes to zero;
		\item there exist constants $M,N,\hat{r} > 0$ such that for all $\lambda \in \C$ with $\Re(\lambda) \geq -N$ and for all $r \in (0,\hat{r})$
		\[
		\|R_{i}(\lambda;r)\| \leq M \text{ for } i=1,2,3.
		\]
	\end{enumerate}
	\end{assumption}
    The implementation errors studied in \Cref{sec:losing_stability} fit this framework:
	\begin{itemize}
		\item $R_{1}(\lambda;r) = e^{-\lambda r} I_m$ and $R_{2}(\lambda;r) = R_{3}(\lambda;r) = I_p$ for \eqref{eq:feedback_delay};
		\item $R_1(\lambda;r) = I_m$, $R_2(\lambda,r)=\begin{cases}
		\frac{1-e^{-\lambda r}}{\lambda r}I_p & \lambda r \neq 0 \\
		I_p & \lambda r = 0
		\end{cases}$ and $R_3(\lambda,r)=I_p$ for \eqref{eq:characteristic_function_finite_difference};
	\end{itemize}
	 We note that also the low-pass filter in \eqref{eq:pid_lowpass_filter} can be interpreted in terms of this perturbation framework: $R_1(\lambda;r) = I_m$, $R_2(\lambda,r)=I_p/\left(\lambda r + 1\right)$ and $R_3(\lambda,r)=I_{p}$.
	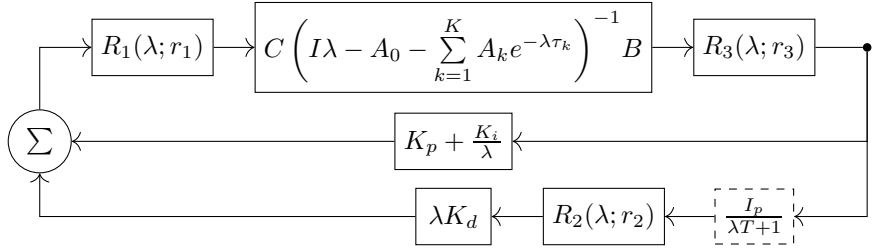
\begin{figure}[!htbp]
		\centering
		\begin{tikzpicture}
			\node[draw,minimum width=1.5cm,minimum height=1cm] (plant) at (0,0)  {$C\left(I\lambda-A_0-\sum\limits_{k=1}^{K}A_k e^{-\lambda \tau_k}\right)^{-1}\!B$};
			\node[draw,circle] (sum) at (-5.5,-1.25) {$\sum$};
			\node[draw,circle,fill=black,inner sep=1pt,minimum size=1pt] (node) at (5.5,0) {};
			\node[draw,minimum width=1cm,minimum height=0.75cm](Kp) at (0,-1.25) {$K_p+\frac{K_i}{\lambda}$};
			\node[draw,minimum width=1cm,minimum height=0.75cm](Rout) at (4,0) {$R_{3}(\lambda;r_3)$};
			\node[draw,minimum width=1cm,minimum height=0.75cm](Rin) at (-4,0) {$R_{1}(\lambda;r_1)$};
			\node[draw,minimum width=1cm,minimum height=0.75cm](Kd) at (0,-2.25) {$\lambda K_d$};
			\node[draw,minimum width=1cm,minimum height=0.75cm](Ro2) at (2,-2.25) {$R_2(\lambda;r_2)$};
			\node[draw,dashed] (filter) at(4,-2.25) {{$\frac{I_p}{\lambda T + 1}$}};
						
			\draw[-{>[scale=1.4]}] (sum) |- (Rin.west);
			\draw[-{>[scale=1.4]}] (Rin) -- (plant);
			\draw[-{>[scale=1.4]}] (plant)--(Rout);
			\draw[-{>[scale=1.4]}] (node)|-(Kp);
			\draw[-{>[scale=1.4]}] (node)|-(filter);
			\draw[-] (Rout) -- (node);
			\draw[-{>[scale=1.4]}]  (Ro2.west) -- (Kd);
			\draw[-{>[scale=1.4]}] (filter) -- (Ro2.east);
			\draw[-{>[scale=1.4]}] (Kp.west) -- (sum);
			\draw[-{>[scale=1.4]}] (Kd) -| (sum); 
		\end{tikzpicture}
		\caption{Closed-loop description of the considered perturbation framework.}
		\label{fig:strong_stability_framework}
	\end{figure}
	
Next, in analogy with neutral functional differential equations, in which stability can be sensitive to arbitrarily small perturbations on the delays, we define strong stability as follows.
\begin{definition}
	A closed-loop system is \emph{strongly stable\/} if for every set $\{R_{i}(\lambda;r_{i})\}_{i=1}^3$ that satisfies the five assumptions given above, there exists a $\hat{r} > 0$ such that the corresponding perturbed closed-loop system is exponentially stable for all $r_{1}$, $r_{2}$ and $r_{3}$ in the open interval $(0,\hat{r})$.
\end{definition}
\begin{remark}
	The notion of strong stability used here should not be confused with the one used in \cite[Section 5.3]{Vidyasagar2011}, where strong stability refers to stabilization with controllers that themselves are stable.
\end{remark}
In the following proposition, it is shown that a stable closed-loop system of the form \eqref{eq:sys}-\eqref{eq:pid_control} can be made strongly stable by including a low-pass filter with a sufficiently large cut-off frequency if the low-pass filter itself is not destabilizing. 

\begin{proposition}
	\label{prop:strong_stability}
	Assume that the gain matrices $K_p$, $K_d$ and $K_i$ are such that the closed-loop system \eqref{eq:sys}-\eqref{eq:pid_control} is exponentially stable. If $B K_d C - I_n$ is Hurwitz, then the closed-loop of \eqref{eq:sys} and \eqref{eq:pid_lowpass_filter} is strongly stable for sufficiently small~$T$.
\end{proposition}
\begin{proof} 
By incorporating the perturbations  $\left\{R_{i}(\lambda;r_{i})\right\}_{i=1}^{3}$ in \eqref{eq:characteristic_function_lowpass_filter}, the characteristic function becomes
\[
\begin{array}{>{$}p{.9\textwidth}<{$}}
H_4(\lambda;T,r_1,r_2,r_3) :=\det
\Bigg(
\lambda
\begin{bmatrix}
I_n  & 0\\
0 & I_q
\end{bmatrix}
- 
\begin{bmatrix}
A_0 & 0 \\
0 & 0
\end{bmatrix}
-  \displaystyle\sum_{k=1}^{K} \begin{bmatrix}
A_k & 0 \\
0 & 0
\end{bmatrix}e^{-\lambda \tau_k} \\
\hfill
-\begin{bmatrix}
BR_1(\lambda;r_1)& 0 \\
0 & I_q
\end{bmatrix}
\begin{bmatrix}
K_p+\frac{\lambda}{\lambda T+1} K_d R_2(\lambda;r_2) & U_i \\
V_i & 0
\end{bmatrix}
\begin{bmatrix}
R_3(\lambda;r_3) C & 0 \\
0 & I_q
\end{bmatrix}
\Bigg)
\end{array}
\]
We choose numbers $N$ and $\hat{r}$ according to item 5 of \Cref{ass:perturbations}. Furthermore, if $BK_d C-I_p$ is Hurwitz and the closed-loop \eqref{eq:sys}-\eqref{eq:pid_control} is stable, then it follows from \Cref{prop:low_pass} that we can choose a sufficiently small $T$ such that there exists a $c\in(0,N)$ for which $H_3(\lambda;T)$ has no root in $V := \{\lambda \in \C : \Re(\lambda) > -c \}$. Subsequently, for all $r_1$, $r_2$ and $r_3$ in the interval $(0,\hat{r})$, the characteristic roots of $H_4(\lambda;T,r_1,r_2,r_3)$ in $V$ are bounded in modulus by
\[
\begin{array}{>{$}p{.9\textwidth}<{$}}
|\lambda_{0}| \leq \sup_{\lambda \in V} \|A_0\| + \sum_{k=1}^{K} \|A_k\| e^{-\Re(\lambda)\tau_k} + \\[7pt] \hfill  \left\|\begin{bmatrix}
BR_1(\lambda;r_1)& 0 \\
0 & I_q
\end{bmatrix}
\begin{bmatrix}
K_p+\frac{\lambda}{\lambda T+1} K_d R_2(\lambda;r_2) & U_i \\
V_i & 0
\end{bmatrix}
\begin{bmatrix}
R_3(\lambda;r_3) C & 0 \\
0 & I_q
\end{bmatrix} \right\|:= \Xi < \infty.
\end{array}
\]
By using Rouch\'e's theorem as in \Cref{prop:low_pass}, it follows that for sufficiently small $r_1$, $r_2$ and $r_3$ the functions $H_4(\cdot;T,r_1,r_2,r_3)$ and $H_3(\lambda;T)$ have the same number of roots in the right half-plane $V$, namely zero. 
\end{proof}

Note that if $CB=0_{m\times p}$, \ie{}the relative degree\footnote{By interpreting the internal dynamics of system \eqref{eq:sys} as an infinite dimensional ordinary differential equation on the head-tail state space $\mathcal{X} := \R^{n} \times \mathcal{C}([-\tau_K,0],\R^{n})$ equipped with the inner product $\big\langle(x_1,\phi_1),(x_2,\phi_2)\big\rangle_{\mathcal{X}} = x_1^{T}x_2 + \int_{-\tau_K}^{0} \phi_1(s)^{T} \phi_2(s)\, ds$, the relative degree of each input-output channel is defined as in \cite[Definition~1.3]{Morris2007}.} of each input-output channel of system \eqref{eq:sys} is larger than one, then stability implies strong stability even without low pass filtering, as shown in the following proposition.
\begin{proposition}
    Assume that the gain matrices $K_p$, $K_d$ and $K_i$ are such that the closed-loop system \eqref{eq:sys}-\eqref{eq:pid_control} is exponentially stable. If $CB=0_{m\times p}$, then the closed-loop of \eqref{eq:sys} and \eqref{eq:pid_control} is also strongly stable.
\end{proposition}
\begin{proof}
    By incorporating the perturbations $\left\{R_{i}(\lambda;r_{i})\right\}_{i=1}^{3}$ in \eqref{eq:characteristic_function}, the characteristic function becomes
    \begin{equation*}
        \begin{array}{>{$}p{.9\textwidth}<{$}}
        H_5(\lambda;r_1,r_2,r_3) :=
            \det
            \Bigg(
            \lambda
            \begin{bmatrix}
            I_n-BR_1(\lambda;r_1)  K_d R_2(\lambda;r_2) R_3(\lambda;r_3)C & 0\\
            0 & I_q
            \end{bmatrix}-
            \begin{bmatrix}
            A_0 & 0 \\
            0 & 0
            \end{bmatrix}- \\ \hfill
            \displaystyle \sum_{k=1}^{K} \begin{bmatrix}
            A_k & 0 \\
            0 & 0
            \end{bmatrix}e^{-\lambda \tau_k} 
            - \begin{bmatrix}
                BR_1(\lambda;r_1)& 0 \\
                0 & I_q
            \end{bmatrix}
            \begin{bmatrix}
                K_p & U_i \\
                V_i & 0
            \end{bmatrix}
            \begin{bmatrix}
                R_3(\lambda;r_3) C & 0 \\
                0 & I_q
            \end{bmatrix}
            \Bigg).
        \end{array}
    \end{equation*}
    Because $CB=0_{m\times p}$, the Weinstein–Aronszajn identity implies that \[ Q(\lambda;r_1,r_2,r_3) :=
    BR_1(\lambda;r_1)K_d R_2(\lambda;r_2) R_3(\lambda;r_3)C
    \] is nilpotent for all $r_1,r_2,r_3>0$, all $\lambda \in \C$ and all permissible perturbations. This has as a consequence that $I-Q(\lambda;r_1,r_2,r_3)$ is invertible. Further, by the last item of \Cref{ass:perturbations} there exist $\hat{r} > 0$ and $c>0$ such that the norm of $\big(I-Q(\lambda;r_1,r_2,r_3)\big)^{-1}$ is finite in the right half-plane $V:= \left\{\lambda \in \C : \Re(\lambda) > -c \right\}$ for all $r_1$, $r_2$ and $r_3$ in the interval $(0,\hat{r})$. Using a similar approach as in the proof of \Cref{prop:low_pass}, one can show that the modulus of a characteristic root of $H_5(\cdot;r_1,r_2,r_3)$ in the right half-plane $V$ is finite. Furthermore, because $H_5(\cdot;r_1,r_2,r_3)$ uniformly converges to $H_0(\cdot)$ on compact regions in the complex plane, the proposition now follows from a similar argument as in the proof of \Cref{prop:low_pass}.
\end{proof}

Next, we present a condition under which there does not exist a feedback matrix $K_d$ such that the closed-loop system \eqref{eq:sys}-\eqref{eq:pid_lowpass_filter} is strongly stable.
\begin{proposition}
	\label{prop:necessary_condtion_PD}
	If for given gain matrices $K_p$ and $K_i$, and $K_d = 0_{m\times p}$, the characteristic function $H_0(\cdot)$ has an odd number of roots in the closed right half-plane then there does not exist a matrix $K_d$ such that the closed-loop system \eqref{eq:sys}-\eqref{eq:pid_control} is stable and $BK_dC-I$ is Hurwitz; and as a consequence there does not exist a matrix $K_d$ such that the closed-loop system \eqref{eq:sys}-\eqref{eq:pid_lowpass_filter} is strongly stable for all sufficiently small $T$. 
\end{proposition}
\begin{proof}
	We use a similar continuation argument as in \cite[Theorems 5.1]{Michiels2009}. More specifically, assume that $K_d^{\star}$ is such that the closed-loop \eqref{eq:sys}-\eqref{eq:pid_control} is stable. Consider now the feedback law: $u(t) =K_p y(t) + k K_d^{\star} \dot{y}(t) + K_i \int_{0}^{t} y(s)\, ds$ for $k\in[0,1]$. For $k = 0$ the corresponding characteristic function has an odd number of roots in right half-plane, while for $k=1$ there are no right half-plane roots. By increasing $k$ from 0 to 1, all roots are thus moved to the left half-plane. Furthermore, if we vary $k$ these characteristic roots move in complex conjugate pairs as all considered matrices are real-valued. Therefore, to bring an odd number of characteristic roots to the left half-plane, at least one root has to pass through either the origin or infinity. The former is however not possible as a characteristic root in the origin is invariant with respect to changes in $k$ and the closed-loop system is stable for $k=1$. On the other hand, a root crossing through infinity means that there exists a $\hat{k} \in (0,1)$ such that $\det(I-\hat{k}BK_d^{\star}C) = 0$, which implies that $BK_d^{\star}C$ has a real-valued root $1/\hat{k}>1$. The second part of the proposition now follows from \Cref{prop:low_pass}.
\end{proof}

We conclude the subsection with a comment on parametric uncertainty.
\begin{remark} \label{remparametric}
If the control law \eqref{eq:pid_lowpass_filter} achieves strong stability, then the asymptotic stability is robust not only against infinitesimal perturbations satisfying Assumption~\ref{ass:perturbations} and visualized in Figure~\ref{fig:strong_stability_framework} but also against small perturbation on system matrices $A_0,\ldots,A_K$, $B$, $C$ and delays $\tau_1,\ldots,\tau_K$. This result directly follows from the strong stability criterion in Proposition~\ref{prop:strong_stability}.
\end{remark}

\subsection{Illustration of theory using a third-order SISO system}
\label{subsec:illustration_third_order}
In this subsection, we illustrate the results of \Cref{subsec:strong_stability,sec:losing_stability} using the following third-order, single-input single-output system
\begin{equation}
\label{eq:example_3rd_order}
\left\{
\begin{array}{rcl}
\dot{x}(t)& =& \begin{bmatrix}
-1 & 1/3 & 1 \\
1 & 0 & 0 \\
0 & 1 & 0
\end{bmatrix} x(t) + \begin{bmatrix}
2 \\ 0 \\ 0
\end{bmatrix} u(t) \\[20px]
y(t) &=&\begin{bmatrix}
0.5 & 0 & 0.5
\end{bmatrix}x(t)
\end{array}
\right.
\end{equation}
with corresponding transfer function
\[
\frac{s^2+1}{s^3+s^2-(1/3)s-1}.
\]

It is verified in \Cref{appendix:example_3rd_order_PI} that this system can not be stabilized with either P or PI output feedback control. However, as shown next, this system can be stabilized with a PD output feedback controller of the following form
\begin{equation}
\label{eq:third_order_example_controller}
u(t) = k_p y(t) + k_d \dot{y}(t),
\end{equation}
with $k_p \in \R$ and $k_d\in\R$.
The characteristic function $H_0(\cdot)$ reduces in this case to the following polynomial of degree three:
\begin{equation}
\label{eq:example_3rd_order_characteristic_function}
 (1-k_d)\lambda^3+(1-k_p)\lambda^2+\left(-k_d-(1/3)\right)\lambda-k_p-1.
\end{equation}
\Cref{fig:nb_rh_poles} shows the number of right half-plane roots of this polynomial in function of $k_p$ and $k_d$ and can be understood as follows. By the Routh-Hurwitz stability criterion for polynomials of degree three, \eqref{eq:example_3rd_order_characteristic_function} has no roots in the closed right half-plane for $(k_p,k_d)$ in the set 
\[
\{(k_p,k_d): k_d<1, k_p<-1 \text{ and } k_d < \frac{1}{3} + \frac{2}{3} k_p \} \cup \{(k_p,k_d): k_d>1, k_p>1 \text{ and } k_d < \frac{1}{3} + \frac{2}{3} k_p \}.
\]
Further, the number of roots in the closed right half-plane can only change in two ways. Firstly, eigenvalues can cross the imaginary axis. We therefore examine the pairs $(k_p,k_d)$ for which \eqref{eq:example_3rd_order_characteristic_function} has a root on the imaginary axis:
\begin{equation*}
\label{example_third_order_imaginary_axis}
-\jmath\omega^{3}(1-k_d)-\omega^{2}(1-k_p)-\jmath\omega(k_d+\frac{1}{3})-k_p-1 = 0.
\end{equation*}
By splitting the real and imaginary part we get:
\[
\begin{aligned}
-\omega^{2}(1-k_p)-k_p-1 &=0 \\
-\omega^{3}(1-k_d)-\omega(k_d+(1/3)) &=0.
\end{aligned}
\]
For $k_p = -1$ and $k_d$ arbitrary, \eqref{eq:example_3rd_order_characteristic_function} thus has an root at the origin ($\omega^{\star}=0$) and for $k_p \in (-\infty,-1) \cup (1,\infty)$ and $k_d = \frac{1}{3} + \frac{2}{3} k_p$, \eqref{eq:example_3rd_order_characteristic_function} has a pair of complex conjugate roots on the imaginary axis at $\omega^{\star} = \pm \sqrt{\frac{-k_p-1}{1-k_p}}$. Furthermore, the crossing direction at these critical parameter values follows from: 
\[
\left.\frac{\partial \Re(\lambda)}{\partial k_p}\right|_{\lambda=\jmath\omega^{\star}} = \Re\left( \frac{-{\omega^{\star}}^2+1}{-3{\omega^{\star}}^{2}(1-k_d)+2(1-k_p)\jmath{\omega}^{\star}-(k_d+1/3)} \right).
\]
Secondly, the number of right half-plane roots can change by roots moving over infinity at $k_d = 1$. More precisely, for $k_d = 1-\epsilon$ and $\epsilon$ sufficiently small, \eqref{eq:example_3rd_order_characteristic_function} has a root at approximately $\frac{-1+k_p}{\epsilon}$. For $k_p>1$, there thus is a root that moves from the right half-plane to the left half-plane as $k_d$ increases from $1-$ to $1+$. For $k_p<1$, this root moves in the opposite direction.

The closed-loop system \eqref{eq:example_3rd_order}-\eqref{eq:third_order_example_controller} is thus stable for $(k_p,k_d)$ inside the set \[\{(k_p,k_d): k_d<1, k_p<-1 \text{ and } k_d < \frac{1}{3} + \frac{2}{3} k_p \} \cup \{(k_p,k_d): k_d>1, k_p>1 \text{ and } k_d < \frac{1}{3} + \frac{2}{3} k_p \}.\] However, from the theory developed in \Cref{sec:losing_stability,subsec:low_pass_filter}, it follows that for $|k_d|>1$ stability is lost under arbitrarily small feedback delay and that for $k_d>1$ stability is lost under both a finite-difference approximation of the derivative and the inclusion of a low-pass filter. Furthermore, it can be shown that when considering these three implementation errors and assuming that these errors are sufficiently small, the stability of the closed-loop system without low-pass filtering is only preserved for $(k_p,k_d)$ in the set
\[
\left\{(k_p,k_d): k_d>-1, k_p<-1 \text{ and } k_d< \frac{1}{3} + \frac{2}{3} k_p \right\}.
\]
If we consider the perturbation class introduced in \Cref{subsec:strong_stability}, it follows from \Cref{prop:strong_stability} that the closed-loop system with low pass filtering is strongly stable if $(k_p,k_d)$ lies in the set, 
\[
\left\{(k_p,k_d):  k_p<-1 \text{ and } k_d< \frac{1}{3} + \frac{2}{3} k_p \right\}.
\] These different stability regions are indicated on \Cref{fig:nb_rh_poles} using hatching.


 \Cref{fig:nb_rh_poles} can also be used to illustrate \Cref{prop:necessary_condtion_PD}.
\begin{itemize}
	\item For $k_p> -1$ and $k_d = 0$, the characteristic function has one root in the right half-plane. For fixed $k_p$, this root can be brought to the left half-plane by increasing $k_d$ until $1<k_d<\frac{1}{3}+\frac{2}{3}k_p$. However, as expected from the proof of \Cref{prop:necessary_condtion_PD}, the root moves from the right to the left half-plane via infinity. As a consequence, for the stabilizing values of $k_d$, the matrix $Bk_dC$ has a real eigenvalue larger than one. The corresponding closed-loop system is hence not strongly stable. Furthermore, adding a low-pass filter to the control loop will not fix this robustness problem as the low pass filter itself is destabilizing.
	\item On the other hand, for $k_p<-1$ and $k_d=0$, the characteristic function has two roots in the right half-plane. For fixed $k_p$, these two right half-plane roots can be brought to the left half-plane via the imaginary axis by decreasing $k_d$ until $k_d<\frac{1}{3}+\frac{2}{3}k_p$. It now follows from \Cref{prop:strong_stability} that the closed-loop system with loop pass filtering is strongly stable.
\end{itemize}  
Finally, the uncontrolled system ($k_p = 0$ and $k_d = 0$) is unstable and has an odd number of right half-plane roots. Thus, in a strongly stabilizing PD controller with low pass filtering the $k_p$ and $k_d$ parameters can be understood as follows. The P part brings an additional root to the closed right half-plane, \ie{}the P part on its own is destabilizing. The D part subsequently moves these two right half-plane roots to the left half-plane via the imaginary axis.
\begin{figure}[!htbp]
	\centering
\begin{tikzpicture}
	\draw[->,gray!60] (-4,0)--(4,0) node[right,black]{$k_p$};
	\draw[->,gray!60] (0,-7/3)--(0,3) node[above,black]{$k_d$};
	\draw[-,black,line width=0.4mm] (-4,1) -- (4,1);
	\draw[dash pattern={on 7pt off 2pt on 2pt off 2pt on 2pt off 2pt},black,line width=0.4mm] (-4,-7/3) -- (-1,-1/3);
	\draw[dash pattern={on 7pt off 2pt on 2pt off 2pt on 2pt off 2pt},black,line width=0.4mm] (1,1)-- (4,3);
	\draw[dash pattern={on 7pt off 2pt on 1pt off 3pt},black,line width=0.4mm] (-1,3) -- (-1,-7/3);
	\draw[dashed,gray!80] (-4,-1) -- (4,-1);
	\node [circle,fill=black,inner sep=2pt,minimum size=2pt,label=right:{\small(-1,-$\frac{1}{3}$)}] at (-1,-1/3) {}; 
	\node[circle,fill=black,inner sep=2pt,minimum size=2pt,label=above right:{\small(-1,1)}] at (-1,1) {};
	\node[circle,fill=black,inner sep=2pt,minimum size=2pt,label=below:{\small(1,1)}] at (1,1) {};

	\node[draw=gray!60,circle] (rhp2) at (-2,0.5) {2};
	\node[draw=gray!60,circle] (rhp22) at (0.5,2) {2};
	\node[above,yshift=2ex,rotate=33.7,line width=0.3mm] (w0) at (-3/2,-2/3) {$\omega=0$};
	\node[above,yshift=2ex,rotate=33.7,line width=0.3mm] (w0) at (-4,-7/3) {$\omega=1$};
	\draw[->,yshift=1ex,xshift=-1ex] (-3/2,-2/3) -- (-1,-1/3);
	\draw[->,yshift=1ex,xshift=-1ex] (-3.5,-2) -- (-4,-7/3);
	
	\draw[->,red] (0,0.1) -- (-2.5,0.1);
	\draw[->,red] (-2.5,0) -- (-2.5,-1.75);
	
	\node[above,yshift=2ex,rotate=33.7,line width=0.3mm] (w0) at (1,1) {$\omega=\infty$};
	\node[above,yshift=2ex,rotate=33.7,line width=0.3mm] (w0) at (3.5,8/3) {$\omega=1$};
	\draw[->,yshift=1ex,xshift=-1ex] (3/2,4/3) -- (1,1);
	\draw[->,yshift=1ex,xshift=-1ex] (3.5,8/3) -- (4,3);
	
	\node[draw=gray!60,circle] (rhp3) at (-2,2) {3};
	\node[draw=gray!60,circle] (rhp1) at (0.5,-0.5) {1};
	\fill[pattern=dots] (1,1) -- (4,3) -- (4,1) -- cycle;
	\fill[pattern=grid] (-1,-1/3) -- (-1,-1) -- (-2,-1) -- cycle;
	\fill[pattern=south east lines] (-1,-1/3) -- (-1,-7/3) -- (-4,-7/3) -- cycle;
	\node[circle,fill=white,draw=gray!60] (stab2) at (3.5,2) {0};
	\node[circle,fill=white,draw=gray!60] (stab1) at (-2,-1.75) {0};
	\node[rectangle,pattern=dots,minimum width =10pt,minimum height=8pt,label=right:{PD controller is stabilizing but not robust against any perturbation type of \Cref{sec:losing_stability,subsec:low_pass_filter}}] at (-6.7,-2.7) {};
	\node[rectangle,pattern=grid,minimum width =10pt,minimum height=8pt,anchor=center, label=right:{PD controller is stabilizing and robust against the perturbations of \Cref{sec:losing_stability,subsec:low_pass_filter}}] at (-6.7,-3.2) {};
	\node[rectangle,pattern=south east lines,minimum width =10pt,minimum height=8pt,label=right:{PD controller with low pass filtering is strongly stabilizing}] at (-4.4,-3.7) {};
	\draw[-] (-4.1,-4.3) -- (-3.4,-4.3) node[right] {Crossing via "infinity"};
	\draw[dash pattern={on 7pt off 2pt on 1pt off 3pt}] (0.5,-4.3) -- (1.2,-4.3) node[right] {Crossing at origin};
	\draw[dash pattern={on 7pt off 2pt on 2pt off 2pt on 2pt off 2pt}] (-4.1,-4.9) -- (-3.3,-4.9) node[right] {Crossing via imaginary axis at $\omega = \pm \sqrt{\frac{-k_p-1}{-k_p+1}}$};
\end{tikzpicture}
\caption{The number of roots of \eqref{eq:example_3rd_order_characteristic_function} in the closed right half-plane in function of $k_p$ and $k_d$ (grey circles). Transitions between the different regions correspond to eigenvalues crossing via ``infinity" (full line), the origin (dot dashed) or the imaginary axis (dot dot dashed). The parameter pairs for which the PD controller is stable but not robust against any of the three types of perturbations discussed in \Cref{sec:losing_stability,subsec:low_pass_filter}, for which the PD controller is stable and robust against these perturbations, and for which the PD controller with low-pass filtering is strongly stable are indicated using hatching.}
\label{fig:nb_rh_poles}
\end{figure}
\section{Design procedure for strongly stabilizing controllers}
\label{sec:numerical_algorithm}
In this section, we describe a computational procedure to design a PID output feedback controller with a low pass filtering of the derivative signal which strongly stabilizes system \eqref{eq:sys}. To synthesize such controllers, we will minimize the spectral abscissa of \eqref{eq:characteristic_function}, \ie{}the real part of its right-most characteristic root, in function of elements of the feedback matrices $K_p$, $K_d$ and $K_i$, under the constraint that $BK_dC-I_n$ is Hurwitz. This leads to the following design procedure:

\begin{enumerate}
	\item choose initial controller matrices $K_p$, $K_d$ and $K_i$;
	\item if $\alpha(B K_{d} C ) = \max \{\Re(\lambda): \det(I_n\lambda-B K_d C )  = 0\} > 0.9$, then re-scale $K_d$ by $0.9/\alpha(B K_{d}C)$
	\item minimize the spectral abscissa of \eqref{eq:characteristic_function} 
	\[
	\begin{array}{>{$}p{.9\textwidth}<{$}}
	f(K_p,K_d,K_i) = \max_{\lambda\in\mathbb{C}} \Big\{\Re(\lambda):\ \det \Big(\lambda (I_n-B K_d C)- A_0 - \\ \hfill \sum_{k=1}^{K} A_k e^{-\lambda \tau_k} - B K_p C -\frac{B K_i C}{\lambda} \Big)=0  \Big\}
	\end{array}
	\]
	as function of the elements of $K_p$, $K_d$ and $K_i$ under the constraint that
	\[
	\alpha(B K_d C) = \max \{ \Re(\lambda) : \det(\lambda I_m - B K_d C) = 0 \}<1;
	\]
	\item choose a sufficiently small $T$ such that \eqref{eq:sys}-\eqref{eq:pid_lowpass_filter} is exponentially stable.
\end{enumerate}
The constraint in step 3 ensures that the condition in \Cref{prop:strong_stability} is fulfilled. This constraint is handled using a penalty method. 
More precisely, we minimize
\begin{equation}
\label{eq:penalty_method}
f(K_p,K_d,K_i) + t \max\left(0,\alpha(B K_d C) - 1\right),
\end{equation}
in the elements of the control matrices, where $t$ is increased until the resulting $K_d$ fulfills the constraint. Note, however, that this function is in general non-convex and, as a consequence, it might have multiple local optima. Therefore, to avoid ending up at a bad local minimum, we restart the optimization procedure from several initial points.

Next, we illustrate the effectiveness of this method on three example problems.
\begin{example}
\label{ex:design_1}
We start with designing a controller for the system examined in \Cref{subsec:illustration_third_order}. As initial parameters we use $(k_p,k_d) = (1.5,1.2)$ and $t=10^2$. This results in $(k_p,k_d)=(-1.08015,-1.04045)$. It thus follows from \Cref{fig:nb_rh_poles} that the resulting controller with low-pass filtering is strongly stable for sufficiently small $T$. To indicate the importance of rescaling the initial $K_d$ and the constraint in step 3, we redesign the controller starting from the same initial parameters without these components. This results in the controller $(k_p,k_d)=(1.26832,1.01777)$. The closed-loop system without filter is now stable, but stability is lost when adding the low pass filter. \hfill $\diamond$
\end{example}
\begin{example}
\label{ex:design_state_delay}
Secondly, we consider a system of form \eqref{eq:sys} with $n=6$, $m=3$, $p=2$ and $K=3$. The system matrices and delays are given in \Cref{appendix:design_state_delay_system_definition}. The open-loop system has 5 right half-plane poles, the right-most of which are a complex conjugate pair located at $2.607\pm  \jmath 2.144$. After running the design procedure starting from a zero initial controller and $t=10^5$ we obtained the feedback matrices given in \Cref{appendix:design_state_delay_system_definition}. The closed-loop system with low-pass filter ($T=10^{-7}$) is now exponentially stable with a decay rate of 0.1768. Furthermore, the real part of all eigenvalues of $B K_d C$ is smaller than one, thus the closed-loop with low pass filter is strongly stable. \hfill $\diamond$
\end{example}
\begin{example}
\label{ex:design_quadcopter}
	Finally, as last example, we consider the stabilization of a quadcopter around its equilibrium point, \ie{}hovering at a fixed position and a fixed orientation of the principal axes. We use the twelve dimensional linearized model given in \cite[Equations (6.4-6.7)]{Jirinec2011} for the parameters given in \cite[Section 4]{Jirinec2011}. By choosing an appropriate output matrix $C$ we obtain the following system of form \eqref{eq:sys}:
	\begin{equation}
	\label{eq:sys_quadcopter}
	\left\{
		\begin{array}{rl}
		
		\Delta\dot{x}(t) =& \setlength{\arraycolsep}{2.5pt} \begin{bmatrix}
		0_{3\times 3} & I_3 & 0_{3\times 3} & 0_{3\times 3} \\
		0_{3\times 3} & 0_{3\times 3} & \begin{bmatrix}
		0 & -g & 0 \\
		g & 0 & 0 \\
		0 & 0 & 0
		\end{bmatrix} & 0_{3 \times 3} \\
		0_{3\times 3} & 0_{3\times 3} & 0_{3 \times 3} & I_3\\
		0_{3\times 3} & 0_{3\times 3} & 0_{3 \times 3} & 0_{3 \times 3}
		\end{bmatrix} \Delta x(t) + \begin{bmatrix}
		0_{3\times 4} \\
		\begin{bmatrix}
		0 & 0 & 0 & 0\\
		0 & 0 & 0 & 0\\
		\frac{-2b\Theta_0}{m} & \frac{-2b\Theta_0}{m}& \frac{-2b\Theta_0}{m} & \frac{-2b\Theta_0}{m} \\
		\end{bmatrix} \\
		0_{3\times 4} \\
		\begin{bmatrix}
		0 & \frac{2lb \Omega_0}{I_x} & 0 & \frac{-2lb \Omega_0}{I_x}\\
		\frac{2lb \Omega_0}{I_y} & 0 & \frac{-2lb \Omega_0}{I_y} & 0\\
		\frac{-2d \Omega_0}{I_z} & \frac{2d \Omega_0}{I_z} & \frac{-2d \Omega_0}{I_z}& \frac{2d \Omega_0}{I_z}\\
		\end{bmatrix}
		\end{bmatrix} \Delta u(t) \\[55pt]
		y(t) =& \begin{bmatrix}
		I_3 & 0_{3\times 1} & 0_{3\times 1} & 0_{3\times 1} & 0_{3\times 3} & 0_{3\times 1} & 0_{3\times 1} & 0_{3\times 1} \\
		0_{1\times 3} & 1 & 1 & 1 & 0_{1\times 3 } & 0 & 0 & 0 \\
		0_{3\times 3} & 0_{3\times 1} & 0_{3\times 1} & 0_{3\times 1} & I_3 & 0_{3\times 1} & 0_{3\times 1} & 0_{3\times 1} \\
		0_{1\times 3} & 0& 0 & 0 & 0_{1 \times 3} & 0 & 0 & 1
		\end{bmatrix}\, \Delta x(t)
		\end{array}
		\right.
	\end{equation}
	with $g = 9.8\, \mathrm{m\,s^{-2}}$, $m = 1.32\, \mathrm{kg}$, $b = 1.5108\cdot10^{-5}\, \mathrm{kg\,m}$, $l = 0.214\, \mathrm{m}$, $I_x = 9.3\cdot10^{-3}\, \mathrm{kg\,m^{2}}$, $I_y = 9.2\cdot10^{-3}\,\mathrm{kg\,m^{2}}$, $I_z = 151\cdot10^{-2}\,\mathrm{kg\,m^{2}}$, $d = 4.406\cdot10^{-7}\, \mathrm{kg\,m^2\,s^{-1}}$ and $\Omega_0 = \sqrt{(mg/(4b))}$.
	 Starting from ten random initial parameter values and fixing $t$ to $10^2$, resulted in ten different PID controllers with low-pass filtering that each strongly stabilize \eqref{eq:sys_quadcopter} for sufficiently large cut-off frequencies. Next, we examine the controller with the best performance, \ie{}the controller matrices that resulted in the smallest $f(K_p,K_d,K_i)$ and which is given in \Cref{appendix:gain_matrices_quadcopter}, and as cut-off frequency we choose $1/T=10^{6}\ \mathrm{s}^{-1}$. By computing the characteristic roots of \eqref{eq:characteristic_function_lowpass_filter}, we find that the closed-loop system is exponentially stable with a decay rate of $0.7526$. \Cref{fig:spectrum_quadcopter} shows these characteristic roots. As expected, four characteristic roots lie at approximately $\frac{\lambda_i-1}{T}$ with $\lambda_i$ the  eigenvalues of $K_d C B$. The real part of almost all remaining roots is approximately equal to the spectral abscissa, which is a typical phenomenon for the direct optimization framework. \hfill $\diamond$
	 
	 \begin{figure}[!htbp]
	 	\begin{subfigure}{0.45\linewidth}
	 		\includegraphics[width=\linewidth]{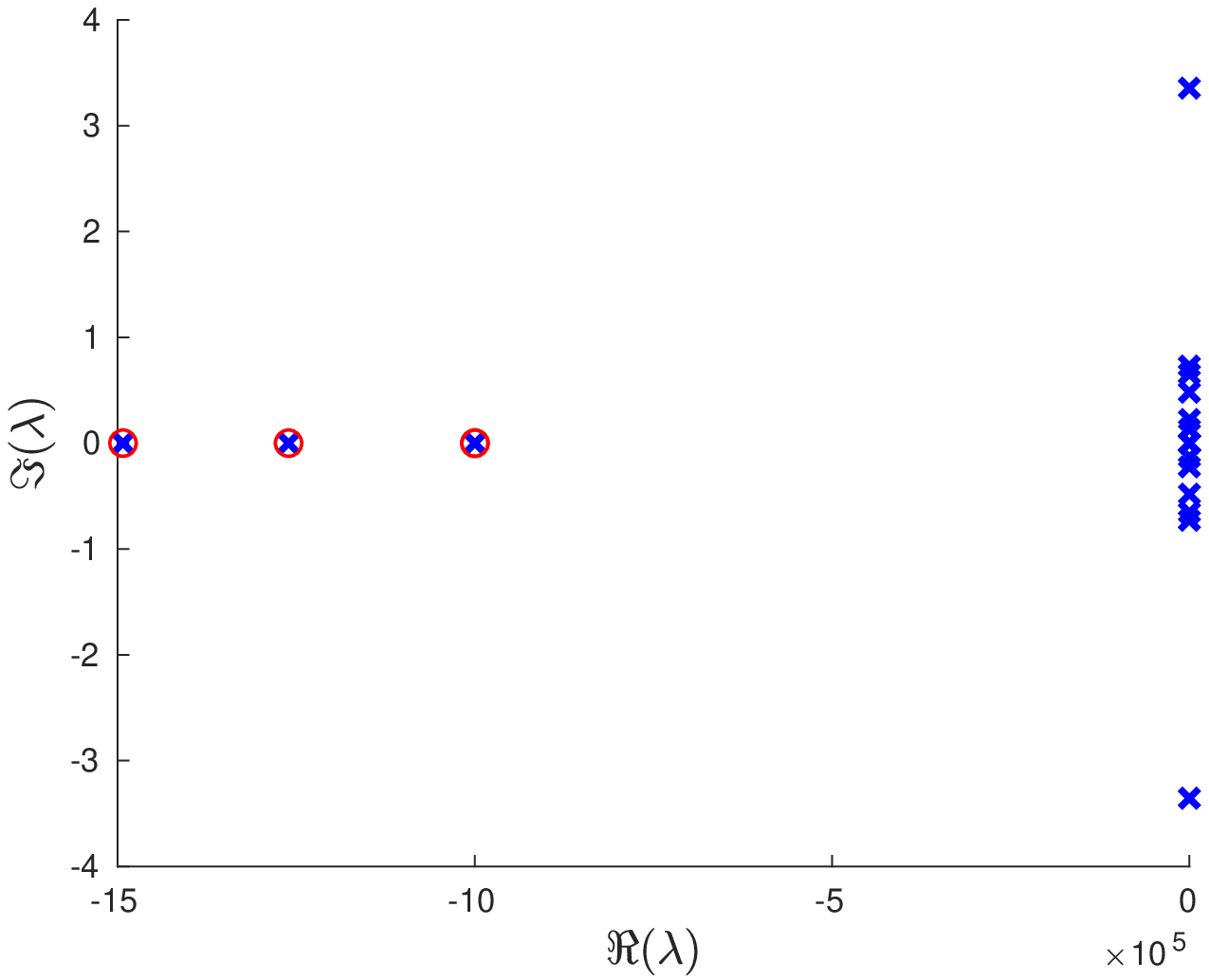}
	 	\end{subfigure}
 		\begin{subfigure}{0.45\linewidth}
 			\includegraphics[width=\linewidth]{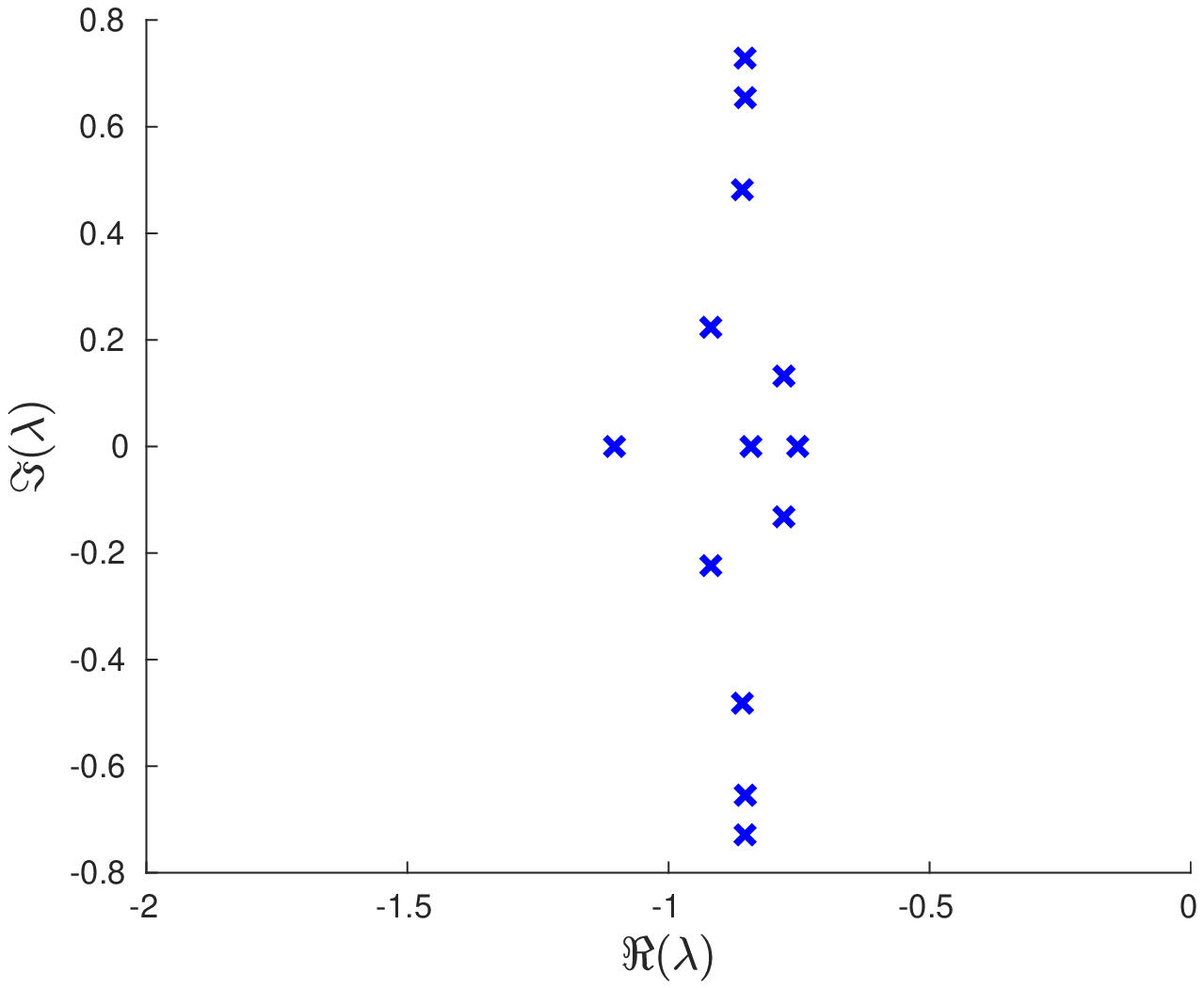}
 		\end{subfigure}
 		\caption{Blue crosses: roots of the instance of characteristic function \eqref{eq:characteristic_function_lowpass_filter} associated with the closed-loop quadcopter system for the controller matrices that minimized $f(K_p,K_d,K_i)$ and $\frac{1}{T}=10^{6}\ \mathrm{s}^{-1}$. In the right-hand panel we zoomed in on the rightmost characteristic roots.\\
 		Red circles: eigenvalues of $\frac{K_d C B-I_{4}}{T}$.}
 		\label{fig:spectrum_quadcopter}
	 \end{figure}
\end{example}
\section{Extensions}
\label{sec:generalization}

We present two extensions of the previous results. First, we consider systems which include an input delay in the nominal model.   
Second, we briefly comment on incorporating  bounded (non-vanishing) perturbations on the system matrices and nominal delays, as so far we dealt with analyzing and resolving the sensitivity of stability with respect to infinitesimal perturbations (see also Remark~\ref{remparametric}).

\subsection{Systems with input delay}
This subsection considers systems with a fixed input delay $\tau_u > 0$,

\begin{equation}
\label{eq:sys_input_delay}
\left\{\begin{array}{l}
\dot x(t) = A_0\, x(t) + \sum_{k=1}^{K} A_k\, x(t-\tau_k) +B\, u(t-\tau_u),\\
y(t)=C\, x(t). 
\end{array}\right.
\end{equation}
The stability of the nominal closed-loop system is now characterized by the roots of \eqref{eq:characeristic_function_delayed_feedback}, where $r$ is fixed to $\tau_u$. It thus follows from a similar argument as in the proof of \Cref{prop:delayed_feedback}, that $\rho(BK_d C)< 1$ is a necessary condition for the exponential stability of the nominal system. In the following proposition, we examine strong stability for systems with input delay. In contrast to systems without input delay, there is now no additional constraint (besides stabilizing the system) on $K_d$ to achieve strong stability using a low pass filter with a sufficiently large cut-off frequency.
\begin{proposition}
\label{prop:strong_stability_input_delay}
	Assume that the gain matrices $K_p$, $K_d$ and $K_i$ are such that the closed-loop system \eqref{eq:sys_input_delay}-\eqref{eq:pid_control} is exponentially stable. Then the closed-loop of \eqref{eq:sys_input_delay} and \eqref{eq:pid_finite_difference} is strongly stable.
\end{proposition}
\begin{proof}
	By incorporating the low-pass filter and the perturbations $\{R_{i}(\lambda;r_i)\}_{i=1}^{3}$, the characteristic function becomes:
	\[
\begin{array}{>{$}p{.9\textwidth}<{$}}
H_6(\lambda;T,r_1,r_2,r_3) :=\det
\Bigg(
\lambda 
\begin{bmatrix}
I_n  & 0\\
0 & I_q 
\end{bmatrix}
- 
\begin{bmatrix}
A_0 & 0 \\
0 & 0
\end{bmatrix}
-  \displaystyle\sum_{k=1}^{K} \begin{bmatrix}
A_k & 0 \\
0 & 0
\end{bmatrix}e^{-\lambda \tau_k} \\
\hfill
\begin{bmatrix}
BR_1(\lambda;r_1)e^{-\lambda \tau_u}& 0 \\
0 & I_q
\end{bmatrix}
\begin{bmatrix}
K_p+\frac{\lambda}{\lambda T+1} K_d R_2(\lambda;r_2) & U_i \\
V_i & 0
\end{bmatrix}
\begin{bmatrix}
R_3(\lambda;r_3) C & 0 \\
0 & I_q
\end{bmatrix}
\Bigg).
\end{array}
\]
	We have already established that the exponential stability of the closed-loop system of \eqref{eq:sys_input_delay} and \eqref{eq:pid_control} requires {$\rho(BK_dC)<1$}. Next, we will show that the condition  $\rho(BK_d C)<1$ implies that the low-pass filter itself is not destabilizing for $R_1(\lambda;r_1) = R_2(\lambda;r_2) = R_3(\lambda;r_3) = I$. If $\rho(BK_d C)<1$, there exist $\epsilon>0$ and $\hat{T}>0$ such that $\rho\left(BK_d C e^{-\lambda \tau_u}\frac{1}{\lambda T+1}\right)<1$ for all $\lambda \in V := \left\{\lambda\in\C:\Re(\lambda)>-\epsilon\right\}$ and all $T\in(0,\hat{T})$. This implies that $I-BK_d C e^{-\lambda \tau_u}\frac{1}{\lambda T+1}$ is invertible on $V$. It now follows from a similar argument as in \Cref{prop:low_pass} that the low-pass filter does not destabilize the system for sufficiently small $T$. Strong stability follows from a similar argument as in \Cref{prop:strong_stability}.
\end{proof}
The proposition above shows that if a PID controller stabilizes the system \eqref{eq:sys_input_delay}, then including a  low-pass filter with a sufficiently large cut-off frequency makes the closed-loop system strongly stable. It thus suffices to design controller matrices $K_p$, $K_d$ and $K_i$ such that the closed-loop system of \eqref{eq:sys_input_delay} and \eqref{eq:pid_control} is stable. Such controller matrices can, for instance, be obtained using the method presented in \cite{Michiels2011}. More precisely, suitable controller matrices are derived by minimizing the spectral abscissa of the closed-loop system under the constraint that $\rho(B K_d C) < 1$. As in \Cref{sec:numerical_algorithm}, this constraint is incorporated in the cost function. However, instead of including a penalty term, a logbarier method is used to enforce that the constraint is fulfilled in each iteration step.

\begin{example}
\label{ex:input_delay}
  We revisit \Cref{ex:design_quadcopter} with a (fixed) feedback delay of $\tau_u = 0.1 \mathrm{s}$. Although $\rho(BK_d C)=0.4925<1$ for the controller obtained in \Cref{sec:numerical_algorithm}, the closed-loop system is unstable for $\tau_u = 0.1 \mathrm{s}$, both with and without filtering. By applying the procedure of \cite{Michiels2011} starting from this controller, we arrive at a stable closed-loop system with an exponentially decay rate (with $T=10
 ^{-6} \mathrm{s}$) of less than 0.0100 and $\rho(BK_d C)=0.9990<1$. However, when starting from another initial controller, created by swapping the initial $K_i$ and $K_d$, we arrived at an exponentially decay rate (with $T=10^{-6} \mathrm{s}$) of 1.1797, illustrating the non-convexity of the optimization problem and the importance of the starting values. The corresponding controller matrices are given in \Cref{appendix:gain_matrices_quadcopter}.
\end{example}

\subsection{Bounded uncertainties on system matrices and state delays}
In this subsection, we generalize the algorithm presented in \Cref{sec:numerical_algorithm} to systems with bounded uncertainties on the system matrices and the state delays. More specifically,  we consider $L$ real-valued matrix uncertainties 
\[
\delta_1 \in \R^{\varrho_1 \times \varsigma_1}, \dots, \delta_L \in \R^{\varrho_L \times \varsigma_L}
\]
that are bounded in Frobenius norm by 1. For notational convenience we denote the combination of uncertainties as $\delta = \big(\delta_1,\dots,\delta_L\big)$, and the set of all admissible uncertainty values as $\mathcal{D}$:
\[
\mathcal{D} = \left\{(\delta_1,\dots,\delta_L): \delta_{l} \in \R^{\varrho_l \times \varsigma_l} \text{ and } \|\delta_l\|_F \leq 1 \text{ for } l=1,\dots,L \right\}
\]
These matrix uncertainties affect the system matrices in an affine way:
\[
\widetilde{R}(\delta) = R + \sum_{l=1}^{L} G_{R,l} \delta_l H_{R,l}
\]
with $R$ the nominal value of the system matrix. The matrices $G_{R,l}$ and $H_{R,l}$ are real-valued and allow to target specific blocks or parameters in the system matrices. The uncertainties on the state delays, $\delta\tau_1$, \dots, $\delta\tau_K$, are real-valued and bounded in absolute value by $\overline{\delta\tau_k}<\tau_k$. The open loop system now becomes
\[
\left\{\begin{array}{l}
\dot x(t) = \widetilde{A}_0(\delta)\, x(t) + \sum_{k=1}^{K} \widetilde{A}_k(\delta) \, x\big(t-(\tau_k + \delta\tau_k)\big) +\widetilde{B}(\delta)\, u(t),\\
y(t)=\widetilde{C}(\delta)\, x(t). 
\end{array}\right.
\]
The goal is to design a PID controller with low-pass filtering that strongly stabilizes the system for all admissible uncertainty values. To this end, the the design procedure of \Cref{sec:numerical_algorithm} has to be modified in the following way. Firstly, the constraint $\alpha(\widetilde{B}(\delta) K_d \widetilde{C}(\delta))<1$ now has to hold for all admissible uncertainty values or in other words,
\begin{equation}
\label{eq:constraint_uncertainty}
\alpha^{\mathrm{ps}}(\widetilde{B} K_d\widetilde{C}) = \max \left\{\Re(\lambda): \exists \delta \in \mathcal{D} \text{ such that } \det\big(\lambda I_m- \widetilde{B}(\delta) K_d \widetilde{C}(\delta)  \big)=0 \right\} < 1.
\end{equation}
Secondly, instead of minimizing the spectral abscissa, we now have to minimize the pseudo-spectral abscissa of the characteristic function, which is defined as the worst-case value of the spectral abscissa over all admissible uncertainty values,
\begin{multline}
    \label{eq:objective_function_uncertainty}
	f(K_p,K_d,K_i) = \max_{\lambda\in\mathbb{C}} \Bigg\{\Re(\lambda): \exists \delta \in \mathcal{D} \text{ and } |\delta\tau_k|<\overline{\delta\tau_k} \text{ for } k=1,\dots K \text{ such that }  \\ \textstyle \det \Bigg(\lambda \big(I_n\minus\widetilde{B}(\delta) K_d \widetilde{C}(\delta)\big) \minus  \widetilde{A}_0(\delta)  
	-\! \sum\limits_{k=1}^{K} \widetilde{A}_k(\delta) e^{-\lambda (\tau_k+\delta\tau_k)}  \minus \widetilde{B}(\delta) K_p \widetilde{C}(\delta) \minus \frac{\widetilde{B}(\delta) K_i \widetilde{C}(\delta)}{\lambda} \Bigg)\!=0  \Bigg\}
\end{multline}

By noting that the characteristic functions in \eqref{eq:constraint_uncertainty} and \eqref{eq:objective_function_uncertainty} can be rewritten
as
\[
\det\left(\lambda\begin{bmatrix}
I_n & 0 & 0 \\
0 & 0 & 0 \\
0 & 0 & 0
\end{bmatrix}-
\begin{bmatrix}
0 & \widetilde{B}(\delta) & 0 \\
0 & - I_m & K_d \\
\widetilde{C}(\delta) & 0 & -I_p
\end{bmatrix}\right)
\]
and
\[
\begin{array}{>{$}p{.9\textwidth}<{$}}
\det\left(
\lambda
\begin{bmatrix}
I_n & 0 & 0 & 0 & 0 & 0 & 0 \\
0   & 0 & 0 & 0 & 0 & 0 & 0 \\
0   & 0 & 0 & 0 & 0 & 0 & 0 \\
0   & 0 & 0 & 0 & 0 & 0 & 0 \\
0 & 0 & 0 & 0 & I_q & 0 & 0 \\
0   & 0 & 0 & 0 & 0 & 0 & 0 \\
I_n & 0 & 0 & 0 & 0 & 0 & 0
\end{bmatrix}
-
\begin{bmatrix}
\tilde{A}_0(\delta) & \tilde{B}(\delta) & \tilde{B}(\delta) & 0 & 0 & 0 & 0 \\
0 & - I_m & 0 & K_p & U_i & 0 & 0 \\
0 & 0 & - I_m & 0 & 0 & K_d & 0\\
\tilde{C}(\delta) & 0 & 0 & - I_p & 0 & 0 & 0 \\
0 & 0 & 0 & V_i & 0 & 0 & 0\\
0 & 0 & 0 & 0 & 0 & - I_p & \tilde{C}(\delta) \\
0 & 0 & 0 & 0 & 0 & 0 & I_n
\end{bmatrix}\right. - \\[45pt] \hfill
\displaystyle\left.\sum_{k=1}^{K}
\begin{bmatrix}
\tilde{A}_k(\delta) & 0 & 0 & 0 & 0 & 0 & 0 \\
0 & 0 & 0 & 0 & 0 & 0 & 0\\
0 & 0 & 0 & 0 & 0 & 0 & 0\\
0 & 0 & 0 & 0 & 0 & 0 & 0\\
0 & 0 & 0 & 0 & 0 & 0 & 0\\
0 & 0 & 0 & 0 & 0 & 0 & 0\\
0 & 0 & 0 & 0 & 0 & 0 & 0
\end{bmatrix} e^{-\lambda (\tau_k+\delta\tau_k)}\right),
\end{array}
\]
respectively, both \eqref{eq:constraint_uncertainty} and \eqref{eq:objective_function_uncertainty} can be computed using a slight modification of the method presented in \cite{borgioli2019,borgioli2020}. Furthermore, explicit expressions for the derivative of these quantities with respect to the elements of $K_p$, $K_d$ and $K_i$, which are required in the optimization process, can be obtained using \cite[Theorem~4.2]{borgioli2020}.

\section{Concluding remarks}
\label{sec:conclusion}
This work analyzed the potential sensitivity of stability for PID control of dynamical systems with multiple state delays with respect to arbitrarily small modelling errors. This led to the introduction of the new notion of \emph{strong stability}, which requires the preservation of stability under sufficiently small perturbations. We showed that, under a certain algebraic constraint on the derivative gain matrix $K_d$, strong stability can be achieved by adding a low-pass filter to the control loop. A computational procedure to design strongly stabilizing PID controllers was provided. A \textsc{Matlab\textsuperscript{TM}} implementation of this algorithm is available from \url{www.}. The code used to generate \Cref{ex:design_1,ex:design_state_delay,ex:design_quadcopter} is available from the same location.
\section*{Acknowledgments}
This work was supported by the project C14/17/072 of the KU Leuven Research Council, by the project G0A5317N of the Research Foundation-Flanders (FWO - Vlaanderen). The authors are also members of the International Research Network (IRN) \emph{Distributed Parameter Systems with Constraints\/} ("Spa-DisCo") financially supported by the French CNRS and various European institutions and universities (including KU Leuven, Belgium) for the period 2018-2021.

\bibliography{main.bib}
\appendix

 \section{Proof that system \eqref{eq:example_3rd_order} is not stabilizable with P or PI control}
 \label{appendix:example_3rd_order_PI}
 
 \textbf{Case I: P control ($k_i=0$)}
 In this case, the characteristic polynomial becomes:
 \[
 \lambda^3+(1-k_p)\lambda^2-\lambda/3-1-k_p.
 \]
 The result follows immediately from the Routh-Hurwitz conditions for polynomials of degree three.
 
\noindent \textbf{Case II: PI control ($k_i\neq 0$)}
 In this case, the characteristic polynomial becomes:
 \[
 \lambda^4+(1-k_p)\lambda^3+(-1/3-k_i)\lambda^2+(-1-k_p)\lambda-k_i.
 \]
 The corresponding Routh-Hurwitz conditions for polynomials of degree four are:
 \begin{align}
 k_p <1 \label{rh_cond1}\\
 k_i <-1/3 \\
 k_p <-1 \label{rh_cond3}\\
 k_i <0 \label{rh_cond4}\\
 2/3-k_i+(4/3) k_p + k_p k_i > 0 \\
 -2/3-2k_p+2k_i-2k_i k_p-4/3 k_p^2 >0 \label{rh_cond6}
 \end{align}
 Condition \eqref{rh_cond6} implies that 
 \[
 -(2/3)-2k_p-4/3k_p^2 > 2 k_i(k_p-1).
 \]
 Furthermore, because $k_p-1<0$ due to \eqref{rh_cond1}, this is equivalent with
 \[
 \frac{-(2/3)(k_p+1/2)(k_p+1)}{k_p-1}<k_i.
 \]
 Because the right hand side of this inequality is strictly positive for $k_p<-1$, it follows from \eqref{rh_cond3} that $k_i$ must be strictly larger then $0$. This contradicts however \eqref{rh_cond4}.
 \section{System definition and resulting feedback matrices for \Cref{ex:design_state_delay}}
\label{appendix:design_state_delay_system_definition}

The system is defined by
\[
A_0 = \begin{bmatrix}
  -0.3430 & 0.6843 & -0.1278 & 0.4078 & -0.7793 & -1.4286 \\
    1.6663 & 0.0558 & -1.4103 & 0.2430 & -1.7622 & -1.1146 \\
   -0.7667 & -1.4018 & 0.6029 & 0.3975 & -1.9355 & 0.9132 \\
    2.6355 & -1.3601 & -0.4569 & -0.1757 & 1.5269 & 0.9764 \\
   -0.0168 & -1.5217 & -0.1397 & -0.3175 & 0.6787 & -1.5769 \\
    0.3042 & 1.0547 & -0.9833 & -1.1016 & -2.2772 & 0.2041
\end{bmatrix},
\]
\[
A_1 = \begin{bmatrix}
   -1.1636 & -0.0632 & -0.0153 & 0.1706 & 1.0161 & 0.3321 \\
    0.4537 & 0.3837 & -1.5704 & 1.0775 & 1.0633 & -1.2500 \\
    0.6882 & -0.1188 & -0.6172 & 0.1081 & -0.9434 & -0.8816 \\
    0.4581 & 0.4896 & -0.7158 & 0.2237 & -0.2411 & -0.2983 \\
    0.9957 & 0.0992 & -0.1938 & 0.4602 & -0.9461 & -0.2692 \\
   -1.0270 & 0.3322 & 0.6574 & 0.5190 & 0.0591 & 0.3468
\end{bmatrix},
\]
\[
A_2 = \begin{bmatrix}
   -0.1620 & -0.7600 & 0.2185 & 0.6680 & -0.0869 &  0.2811 \\
    0.3548 & 1.1226 & -0.1035 & 0.0191 & 0.8869 & -0.1554 \\
   -0.0062 & -0.1772 & -0.5372 & -0.1446 & 0.4768 & -0.2087 \\
    0.1849 & -0.5384 & 0.2619 & 0.4180 & -0.8080 &  0.7966 \\
    1.0530 & -0.2093 & -0.4522 & -0.6498 & -0.9190 &  0.9116 \\
    0.2829 & 0.0122 & -0.0370 & 0.0390 & -0.0724 & -0.0892 \\
\end{bmatrix},
\]
\[
A3 = \begin{bmatrix}
        2.4775 & 1.0257 & 1.5891 & 0.8326 & -1.3239 & 1.7045 \\
   -0.3964 & 1.2449 & -0.4890 & -1.5695 & -0.8312 & 0.5800 \\
   -2.2922 & 2.2344 & -2.0149 & -0.2128 & 0.6352 & 2.5780 \\
    0.1915 & 1.4213 & -0.8818 & 1.5595 & -1.1228 & -1.1024 \\
    0.8404 & 0.6028 & -1.7115 & -0.7325 & -1.8020 & 2.0994 \\
    0.8771 & 1.2225 & -0.6453 & -2.3762 & 2.2157 & 0.2430 
\end{bmatrix},
\]
\[
B = \begin{bmatrix}
   -0.7928 & -0.1048 & 1.7634 \\
    0.3175 & 0.8241 & -0.0609 \\
    0.3013 & -1.0316 & -0.1341 \\
    0.8311 & -1.2573 & 0.5240 \\
   -2.4712 & 0.2176 & 0.0554 \\
    0.4338 & -0.4529 & 1.1995
\end{bmatrix},
\]
\[
C = \begin{bmatrix}
    1.5959 & -0.8172 & 0.7905 & 1.9030 & -0.3477 & -0.9110 \\
   -0.0133 & 0.1929 & 0.4005 & -1.1900 & -0.2924 & -0.5558 
\end{bmatrix},
\]
$\tau_1 = 0.11$, $\tau_2 = 0.21$ and $\tau_3 = 1$.

The designed gain matrices are
\[
K_p = \begin{bmatrix}
2.8615 & 10.5346 \\
3.9254 & 5.3124 \\
-3.9750 & 13.8184
\end{bmatrix}, \quad
K_d = \begin{bmatrix}
1.7447 & -5.2424 \\
    4.7113 & 3.9833 \\
    0.9434 & 12.2535
\end{bmatrix} \text{and}
\quad
K_i = \begin{bmatrix}
0.6519 & 0.0412 \\
    2.3498 & -2.5687 \\
    1.0265 & 1.5849
\end{bmatrix}.
\]
\section{Resulting gain matrices for \Cref{ex:design_quadcopter,ex:input_delay}}
\label{appendix:gain_matrices_quadcopter}
For \Cref{ex:design_quadcopter}, the resulting gain matrices are
\[
K_p = \begin{bmatrix}
 -5.4621 & -37.8386  & 50.3201  &  2.2475 & -81.8761 & -81.3841  & 20.7403  & 20.9480 \\
  -43.1886 & -19.1534 &  28.1863 & -49.3413 &  -5.5520 &  -3.2277 & -23.3179 & -28.9068 \\
  -13.8907 & -20.9904 &  43.0970  &  3.3196 &  66.3708 &  64.3681 &  32.7644 &  44.4112 \\
   17.0573 &   7.0390 &   9.8189 &  73.1320 &  21.4069 &  16.2811 & -22.8678 & -12.1369
\end{bmatrix},
\]
\[
K_d = \begin{bmatrix}
      2.6020 & -35.2924 &  30.7326  &  4.6261 & -26.2670 & -44.9325 &  19.4149 &  11.9269 \\
  -46.6939 & -28.3477 & 24.5842 & -47.1050 & -52.6639  &  0.5357 & -26.7642 & -12.0257 \\
  -25.7283 &  5.1204 & 21.7455 &  16.7965 &  19.4925 &  48.8419 &  43.1543  & -1.4350 \\
   33.9343 & 13.0930 & 26.7705 &  62.2667 &  45.2039 &  -3.5705 & -10.7538  &  8.9672
\end{bmatrix}
\]
and
\[
K_i = \begin{bmatrix}
   20.3716  & -4.9418 &  23.0753  &  9.2952  & -5.5020  & -0.1435 &  10.1706 &  20.7039 \\
  -15.7922 & -15.7094 & -12.4776 & -32.2340  & -5.9862  &  5.4981 & -16.9445 & -23.1467 \\
  -35.2062  &  0.9211 &   1.7405 &   6.6778  &  0.9162  &  1.9864 &  11.1697 &  32.2109 \\
   -6.5286  & 10.7784 & -21.8867  & 36.4701  &  0.9330  & -2.7628 & -24.3316 & -23.2082
\end{bmatrix}. 
\]
For \Cref{ex:input_delay}, the resulting gain matrices are
\[
K_p = \begin{bmatrix}
16.6430 & -63.3128 &  62.1810  &  4.4335 & -66.8876 & -68.0598 &  16.3207 &  16.5152 \\
  -50.8533 & -19.3299 &  30.2406 & -36.1548 &  -9.8565 & -12.9958 & -31.7057 & -26.1467 \\
  -50.0849  &  5.6587 &  42.9924 &  13.2329 &  58.8861 &  52.4494 &  55.4277 &  42.7279 \\
   10.3170  & 15.6469 &   8.9039 &  75.0145 &  19.3511 &  16.4770 & -32.1567 &  -8.2343
\end{bmatrix},
\]
\[
K_d = \begin{bmatrix}
   24.1572 & -13.2747 &  29.8087 &  17.3038 &   4.6970 & -33.3402 &   5.7379 &  -1.0075 \\
  -25.7786 &   4.0052 &  -9.0192 &  -1.2406 & -19.6114 & -13.1057 & -14.1845 &  -4.0257 \\
  -44.9690 &   8.0431 &  14.2946 &   8.2878 &   6.5073 &   0.3612 &   9.4864 &  -1.4291 \\
  -16.8650 &  -2.2785 &   3.3891 &  26.6494 &  11.8299 & -24.9899 & -20.4291 &  -1.1029
\end{bmatrix}
\]
and
\[
K_i = \begin{bmatrix}
3.7628 & -33.6333 &  37.4261 &  13.1179 & -27.1173 & -45.3188 & 15.8870 & 7.5073 \\
  -45.3248 & -37.6528 &  26.1769 & -52.8919 & -50.6522 & 1.5547 & -25.0178 & -20.4135 \\
  -18.7744  & -0.9240 &  14.8563 &   7.1468 &  20.2193 & 49.8381 &  41.4465 &  21.2283 \\
   28.9560 &  27.5731 &  27.6152 &  63.2191 & 43.8716 & -2.5157 &  -6.3298 &  -0.3217
\end{bmatrix}.
\]
\end{document}